\documentclass{amsart}

\title[Semiconcavity, stability and exponential convergence of Sinkhorn's algorithm]{A semiconcavity approach to stability of entropic plans and exponential convergence of Sinkhorn's algorithm }

\author{Alberto Chiarini}
\address{Università degli studi di Padova}
\curraddr{Dipartimento di Matematica ``Tullio Levi-Civita'', 35121 Padova, Italy.}
\email{chiarini@math.unipd.it}
\thanks{}

\author{Giovanni Conforti}
\address{Università degli studi di Padova}
\curraddr{Dipartimento di Matematica ``Tullio Levi-Civita'', 35121 Padova, Italy.}
\email{gconfort@math.unipd.it}

\author{Giacomo Greco}
\address{Università degli studi di Roma Tor Vergata}
\curraddr{RoMaDS - Department of Mathematics, 00133RM Rome, Italy.}
\email{greco@mat.uniroma2.it}
\thanks{GG is supported by the PRIN project GRAFIA (CUP:
E53D23005530006) and MUR Departement of Excellence Programme 2023-2027 MatMod@Tov (CUP: E83C23000330006). GG is grateful to the Department of Mathematics  ``Tullio Levi-Civita'' for its hospitality, where part of this work has been  carried out.}

\author{Luca Tamanini}
\address{Università Cattolica del Sacro Cuore}
\curraddr{Dipartimento di Matematica e Fisica ``Niccol\`o Tartaglia'', I-25133 Brescia, Italy.}
\email{luca.tamanini@unicatt.it}
\thanks{Alberto Chiarini, Giovanni Conforti, Giacomo Greco and Luca Tamanini are associated to INdAM (Istituto Nazionale di Alta
Matematica ``Francesco Severi'') and the group GNAMPA}

\input{./header}

\begin{document}

\setlength{\abovedisplayskip}{10pt} 
\setlength{\belowdisplayskip}{10pt}
\setlength{\abovedisplayshortskip}{10pt}
\setlength{\belowdisplayshortskip}{10pt}
\begin{abstract}
 We study stability of optimizers and convergence of Sinkhorn's algorithm for the entropic optimal transport problem.
 In the special case of the quadratic cost, our stability bounds imply that if one of the two entropic potentials is semiconcave, then the relative entropy between optimal plans  is controlled by the squared Wasserstein distance between their marginals.  When employed in the analysis of Sinkhorn's algorithm, this result gives a natural sufficient condition for its exponential convergence, which does not require the ground cost to be bounded. By controlling from above the Hessians of Sinkhorn potentials in examples of interest, we obtain new exponential convergence results. For instance, for the first time we obtain exponential convergence for log-concave marginals and quadratic costs for all values of the regularization parameter, based on semiconcavity propagation results. Moreover, the convergence rate has a linear dependence on the regularization: this behavior is sharp and had only been previously  obtained for compact distributions  \cite{chizat2024sharper}.  These optimal rates are also established in situations where one of the two marginals does not have subgaussian tails. Other interesting new applications include  subspace elastic costs, weakly log-concave marginals, marginals with light tails (where, under reinforced assumptions, we manage to improve the rates obtained in \cite{eckstein2023hilberts}), the case of Lipschitz costs with bounded Hessian, and compact Riemannian manifolds.

\end{abstract}

\maketitle
\subsection*{Mathematics Subject Classification (2020)}{49Q22, 68Q87, 68W40, 60E15, 90C25}


\section{Introduction}
Given two Polish spaces $\cX,\cY$, marginal probability distributions $\rho \in \cP(\cX), \nu \in \cP(\cY)$, and a cost function $c:\cX\times\cY \rightarrow \mathbb{R}$, the entropic optimal transport problem (EOT) reads as 
\be\label{EOT}\tag{EOT}
\inf_{\pi \in \Pi(\rho, \nu)}  \int_{\cX\times\cY} c\, \De\pi + \varepsilon \,\scrH(\pi | \rho \otimes \nu),
\ee
where $\Pi(\mu,\nu)$ is the set of couplings of $\mu$ and $\nu$, $\scrH$ denotes the Kullback-Leibler divergence (also known as relative entropy), and $\varepsilon>0$ is a regularization parameter.
The study of EOT has greatly intensified since the observation \cite{cuturi2013sinkhorn} that adding an entropic penalty in the objective function of the Monge-Kantorovich problem (corresponding to $\varepsilon=0$ in \eqref{EOT}) leads to a more convex, more regular, and numerically more tractable optimization task, 
thus opening new perspectives for the computation of transport distances in machine learning and beyond, see \cite{peyre2019computational}. Much of the success of entropic regularization techniques in applications can be attributed to the fact that EOT can be solved by means of an exponentially-fast matrix scaling algorithm, Sinkhorn's algorithm, and to the fact that EOT is more stable than the Monge-Kantorovich problem with respect to variations in the cost or marginals. Because of this, considerable efforts have been made over the last decade to turn these intuitions into sound mathematical statements. This has produced many important contributions which we shall discuss in more detail below, nonetheless, several open questions remain. For example, exponential convergence of Sinkhorn algorithm is not well understood when both the marginals' support and the ground cost are unbounded, as it is the case in the landmark example of the quadratic cost with Gaussian marginals. This article aims at showing that semiconcavity bounds for entropic and Sinkhorn potentials play a key role in answering some of these questions. We now provide the reader with some background on EOT, and then proceed to describe our main contributions.

\subsection{Background on EOT and Sinkhorn algorithm} The entropic optimal transport problem is a regularized version of the Monge-Kantorovich problem
\be\label{MK}\tag{MK}
     \inf_{\pi \in \Pi(\rho, \nu)}  \int_{\cX\times\cY} c\,\De \pi\,.
\ee
In the case when $\cX=\cY$ is endowed with a distance $\sfd$ and $c(x,y)=\sfd(x,y)^2$, the optimal value of \eqref{MK} coincides with the squared Wasserstein distance of order two, which we denote by $\bfW_2^2(\mu,\nu)$. 
It is known \cite[Thm 4.2]{Marcel:notes} that under mild integrability conditions on the cost $c$, there exist two functions $\phieot{\nu}:\cX \rightarrow \overline{\mathbb{R}},\psieot{\nu}:\cY \rightarrow \overline{\mathbb{R}}$ with $\overline{\mathbb{R}}=\mathbb{R}\cup\{+\infty\}$, called entropic potentials, such that the unique optimal plan $\pieot{\nu}$ for \eqref{EOT} admits the Radon--Nikod\'ym density
\be\label{eq:phi_psi_dec}
\frac{\De\pieot{\nu}}{\De(\rho\otimes\nu)}(x,y)= \exp\biggl(-\frac{c(x,y)+\phieot{\nu}(x)+\psieot{\nu}(y)}{\varepsilon}\biggr)\,, \quad \rho\otimes\nu \, \text{a.e.}
\ee
For any measurable $\varphi:\cX\rightarrow \overline{\mathbb{R}}$ and $\psi:\cY\rightarrow\overline{\mathbb{R}}$ we set
\be
\begin{aligned}
\Phi^\rho_0(\varphi)(y) &:=-\varepsilon\log\int_{\cX}\exp\biggl(-\frac{c(x,y)+\varphi(x)}{\varepsilon}\biggr) \rho(\De x)\,,\\
\Psi^\nu_0(\psi)(x) &:=  -\varepsilon\log \int_{\cY} \exp\biggl(-\frac{c(x,y)+\psi(y)}{\varepsilon}\biggr) \nu(\De y)\,.
\end{aligned}
\ee
Imposing that a probability measure of the form \eqref{eq:phi_psi_dec} belongs to $\Pi(\rho,\nu)$ yields the following non-linear system, also known as  Schr\"odinger system,
\be\label{eq:schr_syst}
\left\{
\begin{aligned}
\phieot{\nu} &=  -\Psi_0^{\nu}(\psieot{\nu})\,,\\ 
\psieot{\nu}&=-\Phi^{\rho}_0(\phieot{\nu})\,.
\end{aligned}
\right.
\ee
Note that a priori the identities \eqref{eq:schr_syst}  are valid only $\rho$-almost surely and $\nu$-almost surely respectively. However, since $\Psi_0^\nu$ and $\Phi^\rho_0$ are well-defined even outside the supports, we obtain extensions of $\phieot{\nu}$, $\psieot{\nu}$ taking values in $\bbR\cup\{+\infty\}$ such that \eqref{eq:schr_syst} 
holds everywhere on $\cX\times\cY$.
Sinkhorn's algorithm solves \eqref{eq:schr_syst} as a fixed point problem, \ie, it constructs two sequences of potentials $(\varphi^n_\varepsilon, \psi^n_\varepsilon)$ defined through the iterations
\be\label{eq:sink:iterate}
\begin{cases}
\varphi^{n+1}_\varepsilon &=  -\Psi_0^\nu(\psieot{n})\,,\\
\psi^{n+1}_{\varepsilon} &=-\Phi^\rho_0(\phieot{n+1})\,.
\end{cases}
\ee
Typically, the initialization is $\phieot{0}=0$, but other choices are possible. 

\subsection{Entropic stability for optimal plans}
In this article, we consider a broad setting in which we require the target space $(\cY,\mathfrak{g})$ to be a (possibly unbounded) connected Riemannian manifold without boundary endowed with its Riemannian metric $\mathfrak{g}$ whose associated Riemannian distance we denote $\sfd$. To fix ideas, 
one can simply take $\cY=\mathbb{R}^d$ equipped with the standard Euclidean metric. We will often consider this setting when applying the abstract results to concrete examples in \Cref{sec:esempi}. To state our main results, we need  to introduce a notion of $\Lambda$-semiconcavity by saying that a differentiable function $f:(\cY,\mathfrak{g})\rightarrow \mathbb{R}$ is $\Lambda$-semiconcave if for all $z,y\in \cY$ and for any geodesic $(\gamma_t)_{t\in[0,1]}$ such that $\gamma_0=y$ and $\gamma_1=z$ we have
\be\label{def:Lsemiconcave}
    f(z)-f(y)\leq \langle\nabla f(y),\,\dot{\gamma_0}\rangle_{\mathfrak{g}}+\frac{\Lambda}{2}\,\sfd(z,y)^2\,.
\ee
In the Euclidean setting, where geodesics are straight lines, this is equivalent to requiring that for all $z,y\in \cY$  it holds 
\bes
f(z)-f(y)\leq \langle\nabla f(y),\,z-y\rangle+\frac{\Lambda}{2}\,\sfd(z,y)^2\,.
\ees

The first contribution of this article is an abstract $\scrH$-stability bound for optimal plans associated with different marginal distributions. 
For this result to hold, we require to be able to differentiate under the integral sign in \eqref{eq:schr_syst}, thus being able to write the gradient of $\psieot{\nu}$ as a conditional mean.
\begin{assumption}\label{ass:differentiability}
   Assume $c(x,\cdot)$ to be $\cC^1(\cY)$ and that we can differentiate under the integral sign in \eqref{eq:schr_syst}. That is, the formula  
\begin{equation}\label{eq:diff:assumpotion}
\nabla\psieot{\nu}(y)=-\int_\cX\nabla_2 c(x,y)\,\pieot{\nu}(\De x|y)\,,
\end{equation}
holds in the sense of weak derivatives $\nu$-a.e.
\end{assumption}
The identity \eqref{eq:diff:assumpotion} is  well-known and often used the EOT literature. Particularly, it holds (even in a strong sense) in a variety of examples including the ones we are going to discuss in section \Cref{sec:esempi}. For readers' convenience we show at~\Cref{lem_der_under_integral_sign} in \Cref{sec:app:integrability:weak} that  \Cref{ass:differentiability} holds under an integrability condition involving $\rho$, the  cost $c$ and its first derivative w.r.t. $y\in\cY$.

\begin{theorem}[KL stability of optimal plans]\label{thm:entropic_stability}
Let \Cref{ass:differentiability} hold and $\pieot{\nu},\,\pieot{\mu}$  be the unique optimizers in \eqref{EOT} for the set of marginals $(\rho,\nu)$ and $(\rho,\mu)$. If there exists $\Lambda>0$ such that 
\bes
y\mapsto c(x,y)+\psieot{\nu}(y)
\ees
is $\Lambda$-semiconcave uniformly in $x\in\supp(\rho)$, then
\bes
\scrH(\pieot{\mu}|\pieot{\nu})\leq \scrH(\mu|\nu)+\frac{\Lambda}{2{\varepsilon}}\,\bfW_2^2(\mu,\nu)\,.
\ees
\end{theorem}
The key step in the proof of \Cref{thm:entropic_stability} is the observation, made at Lemma \ref{lem:entropy_between_conditionals}, that the semiconcavity assumption  implies the following $\scrH$-bound between conditional distributions 
\bes
\scrH(\pieot{\nu}(\cdot|y)|\pieot{\nu}(\cdot|z)) \leq \frac{\Lambda}{2\varepsilon}\,\sfd^2(y,z)\,.
\ees
The desired conclusion is then obtained using the gluing method, which is often employed in works on the stability of both entropic and classical optimal transport, see \cite{eckstein2021quantitative, divol2024tight, letrouit2024gluing} for example.

\medskip

Let us also mention here that our proof strategy is robust and easily generalizes to settings where we consider different notions of semiconcavity. In particular, given any non-negative function $\omega\colon\cY\times\cY\to \rset$ we say that a differentiable function $f:(\cY,\mathfrak{g})\rightarrow \mathbb{R}$ is $(\Lambda,\omega)$-semiconcave if for all $z,y\in \cY$ and for any geodesic $(\gamma_t)_{t\in[0,1]}$ such that $\gamma_0=y$ and $\gamma_1=z$ we have
\bes
    f(z)-f(y)\leq \langle\nabla f(y),\,\dot{\gamma_0}\rangle_{\mathfrak{g}}+\frac{\Lambda}{2}\,\omega(z,y)\,.
\ees
Then, if $\bfW_\omega$ denotes the generalized Wasserstein functional induced by $\omega$, \ie, for any $\mu,\,\nu\in\cP(\cY)$ 
\bes
\bfW_\omega(\mu,\nu)\coloneqq \inf_{\pi\in\Pi(\mu,\nu)}\int_{\cY\times\cY} \omega(z,y)\,\De\pi(z,y)\,,
\ees
the entropic stability result stated in \Cref{thm:entropic_stability} immediately generalizes to the following
\begin{theorem}[KL stability of optimal plans generalized]\label{thm:entropic_stability:gen}
Let \Cref{ass:differentiability} hold and $\pieot{\nu},\,\pieot{\mu}$  be the  the unique optimizers in \eqref{EOT} for the set of marginals $(\rho,\nu)$ and $(\rho,\mu)$. If there exists $\Lambda>0$ and a non-negative function $\omega\colon\cY\times\cY\to \rset$ such that 
\bes
y\mapsto c(x,y)+\psieot{\nu}(y)
\ees
is $(\Lambda,\omega)$-semiconcave uniformly in $x\in\supp(\rho)$, then
\bes
\scrH(\pieot{\mu}|\pieot{\nu})\leq \scrH(\mu|\nu)+\frac{\Lambda}{2{\varepsilon}}\,\bfW_\omega(\mu,\nu)\,.
\ees
\end{theorem}

This general result covers a diverse range of marginals, whose log-densities might not decay as Gaussians. For instance, it applies to the quadratic cost $c(x,y)=|x-y|^2/2$ with marginals $\rho(\De x)\propto \exp(-|x|^q)\De x$ and $\nu(\De y)\propto \exp(-\sum_{i=1}^d|y_i|^p_p)\De y$ with $q\geq 2$ and $p\in(1,2)$ (with $\omega(z,y)=|z-y|^{1+\frac{p}{q}}$, cf.~\cite{gozlan2025global}).  

\subsubsection*{Literature review: stability} 
Many recent works study the stability  for entropic optimal transport with respect to variations in the marginal inputs. 
In \cite{lagg2022gradient}, we obtain entropic stability bounds in a negative Sobolev norm, for a general class of problems with costs induced by diffusions on Riemannian manifolds with Ricci curvature bounded from below (which includes the quadratic cost on $\bbRD$). The article \cite{eckstein2021quantitative} 
obtains a quantitative H\"older estimate between the Wasserstein distance of optimal plans and that of their marginals. This result applies to the quadratic cost, and to more general costs. Among the assumptions required, there is a transport-entropy inequality for marginals, which relates to Talagrand's inequality, see \ref{eq:TI} below. A more qualitative stability  is proven under mild hypothesis in \cite{ghosal2021stability}. 
On the dual side, Carlier and Laborde show in \cite{carlier2020differential} $\rmL^\infty$-Lipschitz bounds for potentials for multimarginal problems in a bounded setting. The authors of \cite{deligiannidis2021quantitative} succeed in controlling the $\rmL^\infty$-norm of the difference of entropic potentials with the Wasserstein distance of order one between marginals using an approach based on Hilbert's metric. Among their assumptions, there is boundedness of the cost function and compactness of the marginals' supports. Subsequently, \cite{carlier2024displacement} bounds the difference of potentials with the Wasserstein distance of order two of the respective marginals. The norm used to express these bounds depends on the smoothness and boundedness of the cost. For example, the authors show an $\rmL^\infty$-bound between the gradient of entropic potentials provided the cost is bounded with two bounded derivatives. For the quadratic cost, gradients of entropic potentials provide with good approximations of optimal transport maps (see \cite{greco2024thesis,malamut2023convergence} in unbounded settings and \cite{pooladian2021entropic} in semidiscrete settings), thus justifying the growing interest around the stability of \ref{EOT}.   \cite{divol2024tight} establishes Lipschitz bounds between the $\rmL^2$-norm of the difference of the gradients of entropic potentials and the Wasserstein distance of the marginals. Here, the dependence of the Lipschitz constant on the regularization parameter is polynomial, thus improving on earlier results, 
and marginals may have unbounded support. All these results are obtained using a functional inequality for tilt-stable probability measures, see \cite{chen2022localization} and \cite[Lemma 3.21]{bauerschmidt2024stochastic}. Their main assumption is that both entropic potentials have a bounded Hessian. For this reason, these stability results are the closest in spirit to ours and partially inspired our work. Nonetheless, there are some important differences: we only require a bound on the Hessian of one of the two potentials, and the scope of \Cref{thm:entropic_stability} is not restricted to the quadratic cost, but covers general semiconcave costs. Finally, we mention the works \cite{gallouet2022strong, letrouit2024gluing} since, even if they study the stability problem for unregularized optimal transport, they are somewhat connected to this work. In \cite[Prop.\ 8]{gallouet2022strong} the authors establish a link between the optimality gap  and the distance from the optimal plan, provided that the optimal transport map is generated by a $c$-concave potential. In \cite{letrouit2024gluing} covariance inequalities are employed to bound the distance between transport maps. As we shall see in the sequel of this article, and as already understood in \cite{fathi2019proof,chewi2022entropic,conforti2024weak}, these inequalities are valid tools for estimating the concavity properties of entropic potentials,  see also \cite{gozlan2025global} for very recent results in this direction.

\subsection{Exponential convergence of Sinkhorn's algorithm}

Sinkhorn's algorithm admits a primal interpretation as the iterated Bregman projection algorithm for the relative entropy functional, see \cite{benamou2015iterative}. To explain this, we introduce two sequences of plans, called Sinkhorn plans, as follows:
\be\label{densities:along:sink}
\frac{\De\pi^{n,n}}{\De (\rho\otimes\nu) } = \exp\Big(-\frac{c+\phieot{n}\oplus\psieot{n}}{\varepsilon}\Big)\,,\quad \frac{\De\pi^{n+1,n}}{\De( \rho\otimes\nu)} = \exp\Big(-\frac{c+\phieot{n+1}\oplus\psieot{n}}{\varepsilon}\Big)\,.
\ee
In the above, and in the rest of the paper, for given potentials $\varphi:\cX\rightarrow\bbR$, $\psi:\cY\rightarrow\bbR$  we set $\varphi\oplus\psi(x,y)=\varphi(x)+\psi(y)$. 
Then, $\pi^{n+1,n}$ may be viewed as the entropic projection of $\pi^{n,n}$ in the sense of Csisz\'ar \cite{csiszar1975divergence} on the set of plans having first marginal $\rho$. Likewise, $\pi^{n+1,n+1}$ is the entropic projection of $\pi^{n+1,n}$ on the set of plans having second marginal $\nu$. The second marginal $\nu^{n+1,n}$ of $\pi^{n+1,n}$ and the first marginal $\rho^{n,n}$ of $\pi^{n,n}$ are then given by 
\be\label{eq:wrong_marginals} 
\De \nu^{n+1,n}\coloneqq e^{-\nicefrac{(\psieot{n}-\psieot{n+1})}{\varepsilon}}\,\De\nu\,,\quad \De \rho^{n,n}\coloneqq e^{-\nicefrac{(\phieot{n}-\phieot{n+1})}{\varepsilon}}\,\De\rho\,.
\ee
The second general result of this article is a sufficient condition for the exponential convergence of Sinkhorn plans to the optimal plan $\pieot{\nu}$. To state it, let us recall that a probability measure $\nu\in\cP(\cY)$ is said to satisfy a Talagrand's inequality with constant $\tau$, \ref{eq:TI} for short, if 
\be\label{eq:TI}\tag{$\mathrm{TI}(\tau)$}
\bfW_2^2(\mu,\nu) \leq 2\tau\,\scrH(\mu|\nu)\,, \quad \forall \mu \in\cP(\cY)\,.
\ee

\begin{theorem}[Exponential convergence of Sinkhorn's algorithm]\label{cor:exp_conv_from_stab}
Let $c(x,\cdot)$ be $\cC^1(\cY)$ and assume that there exist $\Lambda\in(0,+\infty)$ and $N\geq 2$ such that 
\bes
y\mapsto c(x,y)+\psieot{n}(y)
\ees
is $\Lambda$-semiconcave uniformly in $x\in\supp(\rho)$ and $n\geq N-1$.
\ben[label=(\roman*)] 

\item\label{item:exp_conv_from_stab_i} If there exists $\tau\in(0,+\infty)$ such that $\nu$ satisfies \ref{eq:TI}, then for all $n\geq N-1$
\bes
 \scrH(\pieot{\nu}|\pi^{n+1,n+1})\leq \scrH(\pieot{\nu}|\pi^{n+1,n})\leq  \bigg(1-\frac{\min\{ \tau\Lambda,\,\varepsilon\}}{\min\{ \tau\Lambda,\,\varepsilon\}+\tau\Lambda}\bigg)^{(n-N+2)} \scrH(\pieot{\nu}|\pi^{0,0})\,.
\ees

\item\label{item:exp_conv_from_stab_ii} If there exists $\tau\in(0,+\infty)$ such that $\nu^{n,n-1}$ satisfies \ref{eq:TI} for all $n\geq N$, then 
\bes 
 \scrH(\pieot{\nu}|\pi^{n+1,n+1})\leq \scrH(\pieot{\nu}|\pi^{n+1,n}) \leq \bigg(1-\frac{  \varepsilon }{\varepsilon+\tau \Lambda}\bigg)^{(n-N+1)} \scrH(\pieot{\nu}|\pi^{0,0})
\ees
holds for all $n\geq N$.
\een
\end{theorem}

Examples of measures satisfying the Talagrand transport inequality \ref{eq:TI} can be found for instance in the monograph \cite{bakry2013analysis} and in \cite{gozlan2011new}. This inequality interpolates between a Log-Sobolev inequality and a Poincar\'e inequality (see e.g. \cite[Theorem 9.6.1 and Proposition 9.6.2]{bakry2013analysis}, see also \cite{gozlan2013characterization} for its connection with log-Sobolev inequalities).  Let us just mention here that this class includes log-concave measures and weakly log-concave measures \cite[Remark 1.7]{conforti2024weak}, and it is stable under products \cite[Proposition 9.2.4]{bakry2013analysis} and under bounded perturbations \cite[Corollary 1.7]{gozlan2011new}.

\bigskip

Building upon the generalized stability bound obtained in~\Cref{thm:entropic_stability:gen}, we are also able to generalize the previous theorem to $(\Lambda,\omega)$-semiconcave settings. In view of this let us introduce a generalized version of \ref{eq:TI}: we say that $\nu$ satisfies a the generalised transport inequality \ref{eq:generalized:TI} if for all $\mu\in\cP(\cY)$ it holds
\be\label{eq:generalized:TI}\tag{$\mathrm{TI}_\omega(\tau)$}
\bfW_\omega(\mu,\nu)\leq 2\tau\,\scrH(\mu|\nu)\,.
\ee
\begin{theorem}[Exponential convergence in $(\Lambda,\omega)$-semiconcave settings]\label{thm:generalized:sink}
    Let $c(x,\cdot)$ be $\cC^1(\cY)$ and assume that there exist $\Lambda\in(0,+\infty)$, a non-negative function $\omega\colon\cY\times\cY\to\rset$ and $N\geq 2$ such that 
\bes
y\mapsto c(x,y)+\psieot{n}(y)
\ees
is $(\Lambda,\omega)$-semiconcave uniformly in $x\in\supp(\rho)$ and $n\geq N-1$. 
\ben[label=(\roman*)] 
\item\label{item:exp_conv_from_stab_i:gen} If $\nu$ satisfies \ref{eq:generalized:TI}, 
then for all $n\geq N-1$
\bes
 \scrH(\pieot{\nu}|\pi^{n+1,n+1})\leq \scrH(\pieot{\nu}|\pi^{n+1,n})\leq  \bigg(1-\frac{\min\{ \tau\Lambda,\,\varepsilon\}}{\min\{ \tau\Lambda,\,\varepsilon\}+\tau\Lambda}\bigg)^{(n-N+2)} \scrH(\pieot{\nu}|\pi^{0,0})\,.
\ees
\item\label{item:exp_conv_from_stab_ii:gen} 
 If Sinkhorn iterates $\nu^{n,n-1}$ satisfy \ref{eq:generalized:TI}, uniformly in $n\geq N$, then 
 \bes 
 \scrH(\pieot{\nu}|\pi^{n+1,n+1})\leq \scrH(\pieot{\nu}|\pi^{n+1,n}) \leq \bigg(1-\frac{  \varepsilon }{\varepsilon+\tau \Lambda}\bigg)^{(n-N+1)} \scrH(\pieot{\nu}|\pi^{0,0})\,.
\ees
\een
\end{theorem}

Examples of measures satisfying \ref{eq:generalized:TI} are given at \Cref{subsection:omega:esempi}. It is worth mentioning that Theorem \ref{thm:generalized:sink} extends to various other situations. For example, if instead \ref{eq:generalized:TI} one assumes $\nu$ to satisfy the more general
\be\label{eq:generalized:TI:gamma} 
\bfW_\omega(\mu,\nu)\leq \tau(\scrH(\mu|\nu)+\scrH(\mu|\nu)^\gamma)\,,
\ee
for some parameter $\gamma>0$. When $\gamma>1$, this generalization is straightforward and Part-\ref{item:exp_conv_from_stab_i:gen} and Part-\ref{item:exp_conv_from_stab_ii:gen} yield again an exponential convergence result. 
Instead, in weaker settings, when $\gamma\in(0,1)$, our proof strategy leads to new polynomial convergence results. We show this in \Cref{appendix:numerica:iteration}.
Whenever $\omega(z,y)=|z-y|^p$ for some $p\geq 1$, this generalized transport inequality with $\gamma=\nicefrac12$ is met  for instance whenever $\nu$ satisfies an exponential integrability condition, namely that $\int \exp(\kappa\,\sfd(z_0,y)^{p})\De\nu(y)<\infty$ for some $\kappa>0$ (cf. \cite[Corollary 2.3]{bolley2005villani}).

\subsection{Examples of application}\label{sec:esempi}
We now discuss several applications of Theorem~\ref{cor:exp_conv_from_stab}. We emphasize that this result provides with a general sufficient condition for the exponential convergence of Sinkhorn's algorithm: the examples we present here do not represent its full range of applicability, which deserves to be further explored. 
Finally, at the end of the current section we analyse a different broad range of examples where the  standard $\Lambda$-semiconcavity approach fails, but that can yet be treated with the the more abstract $(\Lambda,\omega)$-semiconcavity strategy, \ie, with \Cref{thm:entropic_stability:gen} and \Cref{thm:generalized:sink}.
\bei 
\item If we write that \emph{``Sinkhorn's algorithm converges exponentially fast  and for $\varepsilon\leq \varepsilon_m$ (explicitly given)  the rate is $1-\kappa(\varepsilon)$"} we always mean that for any $\varepsilon>0$ there is $\kappa'\in(0,1)$ such that for all $n\geq 1$ we have
\be\label{eq:conv_con_kappa'}
\scrH(\pieot{\nu}|\pi^{n+1,n+1})\leq (1-\kappa')^{n}\, \scrH(\pieot{\nu}|\pi^{0,0})
\ee
and that if $\varepsilon\leq \varepsilon_m$, then $\kappa'$ can be taken to be equal to $\kappa(\varepsilon)$. The requirement that $\varepsilon$ is small enough is just to have $\min\{\tau\Lambda,\varepsilon\}=\varepsilon$ when applying Theorem \ref{cor:exp_conv_from_stab}. In all examples, this hypothesis could be dropped, and the convergence rate is always explicit even for large $\varepsilon$, but we prefer to keep it for readability.
\item If we write that \emph{``Sinkhorn's algorithm converges exponentially fast and for   $\varepsilon\leq \varepsilon_m$ (explicitly given)  the \emph{asymptotic} rate is $1-\kappa(\varepsilon)$"} we mean that for any $\varepsilon>0$ \eqref{eq:conv_con_kappa'} holds for some $\kappa'\in(0,1)$ and all $n\geq1$. Moreover, if $\varepsilon\leq \varepsilon_m$, for all $\delta>0$ there exists $N_\delta$ such that \eqref{eq:conv_con_kappa'} holds for $n\geq N_\delta$ and $\kappa'=\kappa(\varepsilon)-\delta$.
\eei
We start our gallery of examples with the case of (weakly) log-concave marginals. To state the next proposition, we introduce the notation $\|\Sigma\|_2$ to denote the operator norm of the matrix $\Sigma$. Moreover, we denote by $|\cdot|$ the standard Euclidean norm on $\mathbb{R}^d$ and by $|\cdot|_p$ the $p$-th norm, $|x|_p=(\sum_{i=1}^d|x_i|^p)^{1/p}$.

\begin{proposition}[Anisotropic quadratic costs and (weakly) log-concave marginals]\label{prop:anisotropic:cost}
 Let 
  $\cX,\,\cY\subseteq\bbRD$ be two open (connected, possibly unbounded) domains of $\bbRD$ endowed with the Euclidean metric 
  and assume that $c(x,y)\coloneqq \langle x-y,\,\Sigma (x-y)\rangle/2$ for some positive definite symmetric matrix $\Sigma$. Further assume that $\nu$ satisfies \ref{eq:TI}.
\begin{enumerate}[label=(\roman*)] 
    \item \label{prop:anisotropic:cost:caff} If $\cX=\cY=\bbRD$ and $\rho(\De x)=e^{-U_{\rho}(x)}\De x$, $\nu(\De y)=e^{-U_{\nu}(y)}\De y$ and there exist    $\alpha_\rho,\beta_\nu\in(0,+\infty)$ such that
    \be\label{conditio:caff}
        \nabla^2 U_{\rho}(x)\succeq \alpha_\rho,\quad         \nabla^2 U_{\nu}(y)\preceq\beta_\nu \quad \forall x,y \in \bbR^d
    \ee
     in semidefinite order, then Sinkhorn's algorithm converges exponentially fast. If we choose $\phieot{0}$ such that 
     $\nabla^2\phieot{0}\succeq (\sqrt{\nicefrac{\alpha_\rho}{\beta_\nu}}-1)\Sigma$,
      then
     for $\varepsilon\leq \tau \,\sqrt{\nicefrac{\beta_\nu}{\alpha_\rho}}\,\norm{\Sigma}_2$ the rate is 
    \be\label{eq:quick_rate}
    1-\frac{\varepsilon}{\varepsilon+\tau\|\Sigma\|_2(\nicefrac{\beta_\nu}{\alpha_\rho})^{\nicefrac12}}\,.
    \ee
    Otherwise, for any general initialization $\phieot{0}$ the expression appearing in~\eqref{eq:quick_rate} represents the asymptotic rate.
    \item    If $\cX=\bbRD$ and $\rho(\De x)=e^{-U_{\rho}(x)}\De x$ and  there exists $\alpha_\rho\in(0,+\infty)$ such that
    \bes
     \nabla^2 U_{\rho}(x)\succeq \alpha_\rho \quad \forall x\in \bbR^d
    \ees
    in semidefinite order, then Sinkhorn's algorithm converges exponentially fast and for $\varepsilon\leq\|\Sigma\|_2\,\sqrt{\nicefrac{\tau}{\alpha_{\rho}}}$  the rate is
    \be\label{eq:rate_log_conc}
    1- \frac{\varepsilon^2\,\alpha_\rho}{\varepsilon^2\,\alpha_\rho+\tau\,\|\Sigma\|_2^2}\,.
    \ee
    
    \item\label{item_weakly_convex_pot} In the quadratic cost case $\Sigma=\Id$, if $\cX=\bbRD$ and  $\rho$ is weakly log-concave of parameters $\alpha_\rho$ and $L$ (as in~\eqref{eq:weak_log_conc}), then Sinkhorn's algorithm converges exponentially fast and for $\varepsilon\leq \alpha_\rho^{-1}\sqrt{\tau(\alpha_\rho+L)}$ the rate is
    \bes
1-\frac{\varepsilon^2\alpha_\rho^2}{\varepsilon^2\alpha_\rho^2+\tau(\alpha_\rho+L)}\,.
\ees
\end{enumerate}
\end{proposition}

\begin{remark}
In fact, in \Cref{prop:anisotropic:cost}-\ref{prop:anisotropic:cost:caff} the rate is as in \eqref{eq:quick_rate} under the weaker (but less readable) assumption $\nabla^2\phizero{\varepsilon} \succeq - \Sigma -\varepsilon\frac{\alpha_\rho}{2}+ \Big(\frac{\varepsilon^2\alpha^2_\rho}{4}+ \frac{\alpha_\rho}{\beta_\nu}\Sigma^2 \Big)^{1/2}$, see p.\pageref{eq:rate:caff}.
\end{remark}

Following \cite{conforti2023Sinkhorn} we say that $\rho$ is weakly log-concave of parameters $\alpha_\rho>0$ and $L\geq0$ if $\rho(\De x) = e^{-U_{\rho}(x)}\De x$, $U_\rho$ is continuously differentiable and 
\be\label{eq:weak_log_conc}
\langle \nabla U_\rho(\hat{x}) - \nabla U_{\rho}(x),\hat x - x\rangle \geq \alpha_\rho |\hat x -x |^2 - |\hat x -x |f_L(|\hat{x}-x|) \quad \forall x,\hat x \in \bbR^d,
\ee
where for any $L\geq0$, $r>0$ we define $f_L(r)=2L^{1/2}\tanh\big((L^{1/2} r)/2\big)$. If $L=0$, $\rho$ is $\alpha_{\rho}$-log-concave. If $L>0,$ then the Hessian of $U_{\rho}$ is lower bounded by $-L$ and $U_\rho$ behaves almost like an $\alpha_\rho$-convex potential for distant points. Weakly log-concave distributions form a rich class of probability measures that may be viewed as perturbations of log-concave distributions. For example, if $U_\rho$ is a double-well potential of the form
\bes
U_\rho(x)= |x|^4-M |x|^2 +C\,,
\ees
for some $M>0$, $C\in\bbR$, then $\rho$ is weakly log-concave. More generally, if $U_\rho=V_\rho+W_\rho$ with $V_\rho$ strongly convex and $W_{\rho}$ Lipschitz with Lipschitz derivative, then $\rho$ is weakly log-concave.  We stress that the scope of Proposition \ref{prop:anisotropic:cost}-\ref{item_weakly_convex_pot} could be widened invoking the results of \cite{chaintron2025propagation}.  This would allow to consider situations when $U_{\rho}$ does not have a Hessian bounded from below, not even from a negative constant. For example, we could cover Lipschitz perturbations of convex potentials. However, to avoid notation overload, we do not pursue this level of generality here.

The previous result applies to the fundamental example of the Euclidean quadratic cost $|x-y|^2$ as well as other costs considered in practical applications like the subspace elastic costs  
\be\label{def:subspace:elastic:cost}
c(x,y)=\frac{|x-y|^2}{2}+\frac\gamma2\,|A^\perp(x-y)|^2\,,
\ee
where $A$ is a fixed $p$-rank matrix and $A^\perp$ its orthogonal projection $A^\perp=\Id - A^\intercal(AA^\intercal)^{-1}A$ (if $A$ satisfies $AA^\intercal=\Id$, then $A^\perp=\Id - A^\intercal A$). More precisely, Proposition~\ref{prop:anisotropic:cost} applies to $\Sigma=\Id+\gamma \,A^\perp $ having norm equal to $\|\Sigma\|_2=1+\gamma \|A^\perp\|_2$. These costs have recently been considered in \cite{pmlr-v202-cuturi23a, klein2024learning} where the authors notice that they tend to promote sparsity for the corresponding optimal transport map. Let us also remark that the convergence rate obtained here for marginals satisfying \eqref{conditio:caff} is tight in $\varepsilon$. Indeed, \cite[Theorem 1.3]{chizat2024sharper} shows that the convergence rate is always bounded from below from $1-\varepsilon$ when considering Gaussian marginals (which  satisfy  \eqref{conditio:caff}).

\medskip

If we consider a marginal $\rho$ satisfying a light-tail condition, then we can address exponential convergence of more general costs. 
Here and below, for a function $u:\cX\times \cY\to \bbR$ we write $\nabla_1 u(x,y)$, $\nabla_1^2 u(x,y)$ for the gradient and the Hessian with respect to the first component and similarly for $\nabla_2 u(x,y)$ and  $\nabla_2^2 u(x,y)$.

\begin{proposition}[$\rho$ with light tails]\label{prop:light:tails}
      Let $\cX=\bbRD$ and $\cY\subseteq\bbRD$ be an open (connected, possibly unbounded) domain of $\bbRD$ endowed with the Euclidean metric and assume that $\rho$ has light tails in the following sense: there exist $C,\delta>0$, $R\geq0$, $L\geq 0$ such that
\be\label{eq:ass:light:tails}
\nabla^2U_\rho(x)\succeq\begin{cases}
    C\,|x|^{\delta}\quad&\text{ for }|x|>R\,,\\
    -L\quad&\text{ for }|x|\leq R\,.
\end{cases} 
\ee
Further assume that $\nu$ satisfies \ref{eq:TI}. Then we distinguish two cases.
\begin{enumerate}[label=(\roman*)] 
     \item If $c(\cdot,\cdot)$ satisfies
    \be\label{eq:ass:hessia:bounded}
 \nabla_1^2c(x,y),\nabla^2_2c(x,y)\succeq \ellc\quad\text{ and }\quad\norm{\nabla_{1}\nabla_{2} c(x,y)}_2\vee\norm{\nabla_{1}^2c(x,y)}_2\leq\Lc\, 
\ee
uniformly in $x\in\bbRD$ and $y\in\cY$, for some constants $\ellc,\,\Lc \in \rset$, then Sinkhorn's algorithm converges exponentially fast and for $\varepsilon$ small enough (as in \eqref{esp:small:light:tails:del:not:0}) the rate is  
\bes
1-\frac{\varepsilon^{2+\nicefrac2\delta}}{\varepsilon^{2+\nicefrac2\delta}+\varepsilon^{1+\nicefrac2\delta}\tau\delL+2\tau\Lc^2\,(\varepsilon^{\nicefrac2\delta}+(L+2)^2\,C^{-\nicefrac2\delta}(\varepsilon+2\delL)^{\nicefrac2\delta})}
\ees
    where $\delL\coloneqq \Lc-\ellc$.
     \item  When $\cX=\cY=\bbR^d$ and $c(x,y)\coloneqq \langle x-y,\,\Sigma (x-y)\rangle/2$ is the (anisotropic) quadratic cost, Sinkhorn's algorithm converges exponentially fast and for $\varepsilon$ small enough (as in \eqref{esp:small:light:tails:del=0:Sigma}) the rate is
    \bes
1-\frac{\varepsilon^2}{\varepsilon^2+2\,\tau\,\|\Sigma\|_2^2\,\Bigl(1+(L+1)^2\bigl[R^2\vee {C^{-\nicefrac2\delta}}\bigr]\Bigr)}\,.
\ees
\end{enumerate}
\end{proposition}

The class of probability measures satisfying the light-tails condition \eqref{eq:ass:light:tails} includes for instance the potentials with tails lighter than Gaussian, e.g.\ $U_\rho(x)=C_0(1+|x|^{2+\delta})$ for which $C=C_0(1+\delta)(2+\delta)$ and $L=R=0$. The convergence stated above applies to a general class of costs including  $p$-costs $c(x,y)=(1+|x-y|^2)^{\nicefrac{p}{2}}-1$, with $p\in(1,2)$, and the STVS (Soft-Thresholding operator with Vanishing Shrinkage) elastic cost (as proposed in \cite{Schreck2015STVS}, see also \cite{klein2024learning}), that is the cost
\bes
c(x,y)=\frac{|x-y|^2}{2}+\gamma^2\,\sum_{i=0}^d \biggl(\mbox{asinh}\biggl(\frac{|x_i-y_i|}{2\gamma}\biggr)+\frac12-\frac12e^{-2\,\mbox{asinh}\left(\frac{|x_i-y_i|}{2\gamma}\right)} \biggr)\,,
\ees 
for which for any $j\in\{1,2\}$ we have
\bes
\frac12\preceq \nabla_j^2c(x,y)=\Id+\frac12\mathrm{diag}\biggl(\frac{|x_i-y_i|}{\sqrt{(x_i-y_i)^2+\gamma^2}}\biggr)-\frac12\preceq 1\,,
\ees
and similarly $\norm{\nabla_1\nabla_2c(x,y)}_2\leq \nicefrac32$, and thus we get a rate independently of the regularizing parameter $\gamma$.
To the best of our knowledge, the only exponential convergence results that apply under the assumptions of Proposition \ref{prop:light:tails} are in \cite{eckstein2023hilberts}. While the assumptions made therein are weaker, we are able to improve on the convergence rate by showing that it is polynomial in  $\varepsilon$.

\medskip

For Lipschitz costs  with bounded Hessian, we  deduce the following convergence result.

\begin{proposition}[Lipschitz costs]\label{prop:lip:costs}
    Let $\cX,\cY\subseteq\bbRD$ be two open (connected, possibly unbounded) domains of $\bbRD$ endowed with the Euclidean metric and assume that there are two constants $\ellc,\,\Lc \in \rset$ such that
\be\label{hp:hessian:bounds:cost:lip}
\ellc\preceq\nabla^2_2c(x,y)\preceq\Lc\quad\text{ and }\quad\nabla^2_1c(x,y)\preceq \Lc\, \quad \forall x,y\in\bbR^d.
\ee
 Further assume that $\nu$ satisfies \ref{eq:TI}.  We distinguish two cases.
    \begin{enumerate}[label=(\roman*)] 
    \item\label{item:lipschitz_nice_rate} Assume that the cost is Lipschitz in $y$, uniformly in $x$, that is,
\be\label{eq:def:c:lip:unif:y}
\sup_{x\in\bbRD}|c(x,y)-c(x,\hat{y})|\leq \mathrm{Lip}_{\infty,2}(c)\,|y-\hat{y}|\, \quad \forall y,\hat{y}\in\bbR^d.
\ee
Then, Sinkhorn's algorithm converges exponentially fast and for $\varepsilon$ small enough (as in \eqref{eps:small:lip2:cost}) the rate is
\bes
1-\frac{\varepsilon^2}{\varepsilon^2+\tau\varepsilon(\Lc-\ellc)+\tau\mathrm{Lip}^2_{\infty,2}(c)}\,.
\ees
        \item\label{item:lipschitz_not_so_nice_rate} Assume that the cost is Lipschitz in $x$, uniformly in $y$, that is,
\be\label{eq:def:c:lip:unif:x}
\sup_{y\in\bbRD}|c(x,y)-c(\hat{x},y)|\leq \mathrm{Lip}_{1,\infty}(c)\,|x-\hat{x}|\ \quad \forall x,\hat{x}\in\bbR^d\,.
\ee
Further assume that $\norm{\nabla_{12}^2c(x,y)}_2\leq \Lc$ and that $\rho$ satisfies the logarithmic Sobolev inequality \ref{eq:LSI}. Then, Sinkhorn's algorithm converges exponentially fast and for $\varepsilon$ small enough (as in \eqref{eps:small:lip1}) the rate is
\bes
 1-\frac{\varepsilon^4}{\varepsilon^4+\tau\varepsilon^3(\Lc-\ellc)+2\tau\Lc^2C_\rho(\varepsilon^2+4C_\rho\mathrm{Lip}_{1,\infty}^2)}\,.
\ees
    \end{enumerate}
\end{proposition}

We recall here that $\rho$ satisfies a logarithmic Sobolev inequality with constant $C_{\rho}$, \ref{eq:LSI} for short, if for any probability measure $\frp\in\cP(\cX)$ we have
\be\label{eq:LSI}\tag{$\mathrm{LSI}(C_{\rho})$}
\scrH(\frp|\rho)\leq \frac{C_\rho}{2}\int\abs{\nabla\log\frac{\De \frp}{\De\rho}}^2\De \frp\,.
\ee
Many probability measures satisfy LSI, for instance that is the case if $U_\rho$ is $C_{\rho}^{-1}$-convex, or weakly asymptotically convex (see \cite{bakry2013analysis} for more general measures). As a direct consequence of Proposition \ref{prop:lip:costs}, when considering symmetric costs that are Lipschitz, one obtains the exponential convergence of Sinkhorn's algorithm with a rate that is quadratic in $\varepsilon$, without any  assumption other than  \ref{eq:TI} for $\nu$. We are not aware of any exponential convergence result valid under the assumptions of \ref{item:lipschitz_nice_rate}. On the contrary, if we consider the setting of \ref{item:lipschitz_not_so_nice_rate}, then the results of \cite{eckstein2023hilberts} apply, since both \ref{eq:TI} and \ref{eq:LSI} imply Gaussian tails. However, we improve the dependence of the rate of convergence in $\varepsilon$, which we show to be polynomial. 

\bigskip

Theorem \ref{cor:exp_conv_from_stab} can also be employed in the compact setting, leading to the following.

\begin{proposition}[$\supp(\rho)$ compact]\label{prop:compact:supp:sink}
    Assume that $\rho$ has compact support and that $\nu$ satisfies \ref{eq:TI}.
    \begin{enumerate}[label=(\roman*)] 
        \item\label{item:comp_supp_i} If $\nabla_2^2c(x,y)=\Sigma$, then Sinkhorn's algorithm converges and for $\varepsilon\leq \sqrt{\tau}\,\norm{g}_{\rmL^\infty(\rho)}$   the rate is 
        \be\label{eq:comp_supp_rate_1}
        1-\frac{\varepsilon^2}{\varepsilon^2+\tau \|g\|_{\rmL^\infty(\rho)}^2}\,,
        \ee
         where $g(x)\coloneqq \nabla_2 c(x,y)-\Sigma y$. 
        \item\label{item:comp_supp_ii} If $\nu$ is  compactly supported, then Sinkhorn's algorithm converges and for $\varepsilon$ small enough (as in \eqref{eps:small:compact:both}) the rate is
        \bes
1-\frac{\varepsilon^2}{\varepsilon^2+\tau\varepsilon\,(\Lc-\ellc)+\tau\norm{\nabla_2 c}_{\rmL^\infty(\rho\times\nu)}^2}\,,
\ees
where $\ellc,\,\Lc$ are defined such that $\ellc\preceq \nabla_2^2c(x,y)\preceq \Lc$ on $\supp(\rho)\times\supp(\nu)$.
    \end{enumerate}
\end{proposition}
As a straightforward consequence, in the case $c(x,y)=\nicefrac{|x-y|^2}{2}$, if $\supp(\rho)\subseteq B_R(0)$, then the convergence rate is  $1-\nicefrac{\varepsilon^2}{(\varepsilon^2+\tau R^2)}$. The only exponential convergence results with polynomial rates in a setting comparable to that of Proposition  \ref{prop:compact:supp:sink}-\ref{item:comp_supp_i} are those obtained in \cite[Theorem 1.2]{chizat2024sharper}. The authors impose that one the two marginals has a density bounded above and below and that its support is compact and convex. This hypothesis is stronger than assuming that $\nu$ satisfies a Talagrand's inequality. Indeed, it implies LSI and hence TI, cf.\ \cite[Proposition 5.3]{Bobkov2000FromBrunn} and discussion therein, and \cite[Proposition 5.1.6]{bakry2013analysis}. About convergence rates, in \cite{chizat2024sharper} the authors obtain a rate which behaves like $\varepsilon^2$ for small $\varepsilon$, which is the same behavior of \eqref{eq:comp_supp_rate_1}. Moreover, they also show that the asymptotic rate is of order $\varepsilon$. However, this result is not obviously comparable to the findings of Proposition \ref{prop:compact:supp:sink}, as the number of iterations after which this rate is shown to hold diverges as $\varepsilon\rightarrow 0$. The findings of Proposition \ref{prop:compact:supp:sink}-\ref{item:comp_supp_ii} are to be compared with those of \cite[Theorem 3.1-(A1)]{chizat2024sharper}. As before, we can drop the assumption that the support is convex and the same considerations made in the previous case regarding convergence rates hold.

\medskip

Lastly, we discuss the convergence of Sinkhorn's algorithm on Riemannian manifolds, starting from the $d$-dimensional Riemannian sphere $\mathbb{S}^d \subseteq \mathbb{R}^{d+1}$ endowed with the angular metric and corresponding angular distance between two points, that is defined for any $x,\,y \in \mathbb{S}^d$ as
\bes
\sfd(x,y) = \arccos(\langle x,y \rangle)\,.
\ees
On $\sSd$ we will consider the \ref{EOT} problem associated to two possible cost functions
\bes
c(x,y)=1-\cos(\sfd(x,y))=1-\langle x,y\rangle\,,\quad\text{ and }\quad
c_\delta(x,y) = \arccos(\delta\langle x,y\rangle)^2
\ees
with $\delta\in(0,1)$ fixed. Notice that for $\delta=1$ we have $c_\delta(x,y)=\sfd(x,y)^2$, so that $c_\delta$ can be seen as a smoothed version of the cost $\sfd^2$ which avoids the singularity the Hessian of $\sfd^2$ faces when considering  antipodal points $(x,-x)$. For this class of problems we can prove the following.

\begin{proposition}[Riemannian sphere]\label{prop:manifolds}
    Assume that $\nu$ satisfies \ref{eq:TI}. We distinguish two cases.
    \begin{enumerate}[label=(\roman*)] 
        \item When considering the regular cost $c(x,y)=1-\langle x,y\rangle$, Sinkhorn's algorithm converges exponentially fast and for 
     $\varepsilon\leq\tau+\sqrt{\tau+\tau^2} $ the exponential convergence rate equals 
     \bes
     1-\frac{\varepsilon^2}{\varepsilon^2+2\tau\varepsilon+\tau}\,.
     \ees
        \item When the cost considered is $c_\delta(x,y)$, Sinkhorn's algorithm converges exponentially fast and for $\varepsilon$ small enough (as in \eqref{eps:small:sfera:delta}) the rate of convergence is 
\bes
1-\frac{\varepsilon^2}{\varepsilon^2+2\tau\varepsilon\left(\delta^2+\frac{2\pi}{\sqrt{1-\delta^2}}\right)+\frac{4\pi^2\tau}{1-\delta^2}}\,.
\ees
    \end{enumerate}
\end{proposition}

As a final example, we show that our result also applies to the classical Schr\"odinger problem \cite{Schr32, LeoSch} on \emph{any} compact smooth Riemannian manifold $(M,g)$. This problem is usually written as the entropy minimization
\be\label{eq:SP}
\inf_{\pi \in \Pi(\rho,\nu)}\scrH(\pi|R_{0\varepsilon})
\ee
where $R_{0\varepsilon}(\De x\De y) = \msp_\varepsilon(x,y){\textrm{vol}}(\De x){\textrm{vol}} (\De y)$, $\msp_\varepsilon(x,y)$ being the heat kernel (namely the unique solution  at time $\varepsilon$ of the `heat equation' $\partial_t u = \frac12 \Delta u$ for the initial condition $u(0,x)=\delta_x$) and ${\textrm{vol}}$ the volume form. As explained for instance in \cite[Section 3]{lagg2022gradient}, \eqref{eq:SP} can be recast as an \ref{EOT} problem with cost function
\[
c_\varepsilon(x,y) = - \varepsilon\log\msp_\varepsilon(x,y)\,.
\]
Since upper and lower bounds on the sectional curvatures of $M$ allow to control the Hessian of $\log\msp_t$, we obtain the following

\begin{proposition}[Heat kernel cost]\label{prop:manifoldsSP}
Let $M$ be a compact smooth Riemannian manifold and assume that $\nu$ satisfies \ref{eq:TI}. Then for every $\varepsilon \in (0,1]$ Sinkhorn's algorithm converges exponentially fast. Moreover, as soon as $\varepsilon$ is small enough (cf.~\eqref{eps:small:manifold:gen}) the exponential convergence rate equals
\bes
1-\frac{\varepsilon^2}{\varepsilon^2+  2\tau C'(\varepsilon+{\textrm{diam}}(M)^2) + \tau C(\varepsilon+\kappa^-\varepsilon(1+ \varepsilon)+1) }
\ees
where $\kappa \in \mathbb{R}$ denotes a lower bound for the Ricci curvature and $C,C'$ are constants depending on the dimension, Ricci and sectional curvatures of $M$. Finally, if $M$ has non-negative Ricci curvature, then the exponential convergence of Sinkhorn's algorithm holds for any regularizing parameter $\varepsilon>0$.
\end{proposition}

\subsubsection{A $(\Lambda,\omega)$-semiconcavity example}\label{subsection:omega:esempi}
We consider again $\cX=\cY=\bbRD$ with the quadratic cost $c(x,y)=|x-y|^2/2$ and two absolutely continuous marginals with densities $\rho(\De x)\propto \exp(-|x|^q-\delta|x|^2)\De x$ and $\nu(\De y)\propto \exp(-\min\{|y|^2,|y|^p_p\})\De y$ with $q\geq 2$ and $p\in(1,2)$ (recall that $|y|_p=(\sum_{i=1}^d|y_i|^p)^{1/p}$).
By construction $\nu$ is equivalent to $\tilde{\nu}(\De y)\propto e^{-|y|^p_p}$ but its negative log-density is globally semiconcave. In particular, the second derivative does not blow up at $y=0$. One of the challenges of the current setup is that $\nu$ does not satisfy Talagrand's inequality \ref{eq:TI}
but rather the weaker transport inequalities \ref{eq:generalized:TI} and \eqref{eq:generalized:TI:gamma} (as shown in \cite{gentil2005modified}). To profit from this inequality, we need to show that $\psi$ is not simply semiconcave, but its gradient behaves like a H\"older function for distant points. We do so relying on the results of \cite{gozlan2025global}. As we clarify below, not even a polynomial rate of convergence for optimal plans was known in the setting considered here.
\begin{proposition}\label{prop_heavy_sink}
     Let $\cX=\cY=\bbRD$ and consider the quadratic cost $c(x,y)=|x-y|^2/2$ and two absolutely continuous marginals with densities $\rho(\De x)\propto \exp(-|x|^q-\delta|x|^2)\De x$ and $\nu(\De y)\propto \exp(-\min\{|y|^2,|y|^p_p\})\De y$ with $q\geq 2,\delta>0$, $p\in(1,2)$ and $\nicefrac1p+\nicefrac1q\leq 1$. Starting with $\psieot{0}=0,$ Sinkhorn's algorithm converges exponentially fast and there exists $K>0$ independent of $\varepsilon$ such that the rate is 
    \begin{equation*}
     1-\frac{\min\{ K,\,\varepsilon\}}{\min\{K,\,\varepsilon\}+K}
    \end{equation*}
\end{proposition}

Finally, let us conclude mentioning that in here we have focused on a specific example of application of our $(\Lambda,\omega)$-semiconcavity argument for the sake of clarity. Nevertheless, \Cref{thm:entropic_stability:gen} and \Cref{thm:generalized:sink} could be applied in a broader framework.

\subsection{Literature review: Sinkhorn's algorithm}

\subsubsection*{Bounded costs.} Sinkhorn's algorithm has a long history, going back at least to the works of Sinkhorn \cite{Sinkhorn64} and Sinkhorn and Knopp \cite{SinkhornKnopp67} in the discrete setting. Here, it is employed as an algorithm to construct a matrix with prescribed rows and columns sums. Other important early contributions include \cite{Franklin89hilbert,borwein1994entropy}. In particular, \cite{Franklin89hilbert} introduced Hilbert's metric as a tool to prove exponential convergence. The realization \cite{cuturi2013sinkhorn}  that \ref{EOT} provides with a numerically more tractable version of the Monge-Kantorovich problem triggered an explosion of interest on the subject. For bounded costs, \cite{chen2016hilbertmetric} obtained the first exponential convergence results in the continuous setting using the Hilbert metric approach and  \cite{marinogerolin2020} establishes qualitative $\rmL^p$ convergence  and regularity estimates for Sinkhorn iterates using an optimal transport approach; in particular these results apply to the multimarginal case as well. The exponential convergence of the algorithm in the multimarginal setting is a result of Carlier \cite{Carlier22multisink}. A more probabilistic viewpoint is introduced in \cite{greco2023SinkhornTorus}, where contraction estimates are obtained by means of coupling arguments. The work \cite{chizat2024sharper} shows that exponential convergence takes place for semiconcave bounded costs, under various sets of hypotheses on the marginals. 
The main innovation of this work consists in showing that the exponential rate of convergence deteriorates polynomially in $\varepsilon$. This is in contrast with previous works that exhibited a rate of convergence that decays exponentially in $\varepsilon$.

\subsubsection*{Unbounded costs.} For unbounded costs and marginals, R\"uschendorf \cite{ruschendorf1995convergence} establishes qualitative convergence for Sinkhorn plans in relative entropy. These results are improved in \cite{nutz2021entropic} where the authors manage to show qualitative convergence on both the primal and dual side under mild assumptions on the cost and marginals. The work by L\'eger \cite{leger2021gradient} provides with a insightful interpretation of Sinkhorn's algorithm as a block-coordinate descent algorithm on the dual problem, see also \cite{aubin2022mirror,leger2023gradient}. From this interpretation, it follows that the speed of convergence is at least $n^{-1}$ under minimal assumptions. Polynomial rates of convergence are also established in \cite{eckstein2021quantitative} as a consequence of the above mentioned stability results and subsequently improved in \cite{ghosal2022nutz}, exploiting a symmetrized version of the KL-divergence. It is only very recently that the first exponential convergence results for unbounded costs and marginals have appeared. To the best of our knowledge, the first article containing such results is \cite{conforti2023Sinkhorn}, which studies the quadratic cost. The main result is that if the marginals are weakly log-concave and $\varepsilon$ is large enough, exponential convergence of $\nabla \varphi^n$ to $\nabla\phieot{\nu}$  takes place. This kind of convergence is particularly useful as $\nabla\phieot{\nu}$ approximates the Brenier map in the $\varepsilon\rightarrow 0$ limit, see \cite{lagg2022gradient,pooladian2021entropic}. The proof follows a geometric approach that highlights that semiconvexity and semiconcavity bounds on Sinkhorn potentials can be leveraged to obtain exponential convergence. Subsequently, the article \cite{eckstein2023hilberts} succeeds in constructing versions of the Hilbert metric that are contractive for general unbounded costs. In contrast with \cite{conforti2023Sinkhorn}, exponential convergence is shown for all values of $\varepsilon$. Roughly speaking, the main assumption is that there exist some $p>0$ such that $c(x,y)$ grows no faster than $|x|^p+|y|^p$ and the tails of both $\mu$ and $\nu$ decay faster than $\exp(-r^{p+\delta})$ for some $\delta>0$. When applied to the quadratic cost, this assumptions does not completely cover log-concave distributions and their perturbations, leaving out Gaussian marginals for example.  Over the past few years, a number of relevant contributions  focused on different asymptotic properties of Sinkhorn's algorithm than the speed of convergence. It would be impossible to account for all of them here. Let us just mention \cite{berman2020sinkhorn} for the relation with Monge-Amp\`ere equation, \cite{deb2023wasserstein} for the construction of Wasserstein mirror gradient flows, and \cite{sander2022sinkformers} for construction of a Transformer variant inspired by Sinkhorn's algorithm.

\subsubsection{A more precise comparison for the Euclidean quadratic cost }
Here we compare the convergence rates obtained for the quadratic cost setting in this paper with the most recent exponential convergence rates deduced in \cite{conforti2023Sinkhorn, eckstein2023hilberts, chizat2024sharper}. The forthcoming discussion is summarized in \Cref{table:comparison}.

\begin{table}[h!]
\renewcommand{\arraystretch}{1.5}
\resizebox{.99\textwidth}{!}{
\begin{tabular}{|c|c|c|c|c|}
\hline
\textbf{Assumptions}  & \cite{conforti2023Sinkhorn} &\cite{eckstein2023hilberts}& \cite{chizat2024sharper} & \textbf{Our results ($\nu\in$ \ref{eq:TI})}\\
\hline
\makecell{$\rho$ is $\alpha_\rho$-log-concave\\
$\nu$ is $\beta_\nu$-log-convex}&\makecell{$1-\Theta(\varepsilon)$ \\ {\footnotesize if $\mu,\nu$ are Gaussians} \\
{\footnotesize $\varepsilon>\varepsilon_0$ if also $\nu$ log-concave}} &\diagbox{}{} &\diagbox{}{} & $1-\Theta(\frac{\varepsilon}{\tau}\,\sqrt{\nicefrac{\alpha_\rho}{\beta_\nu}})$\\ 
\hline 
$\rho$ is $\alpha_\rho$-log-concave & \makecell{{\footnotesize $\varepsilon>\varepsilon_0$}\\
{\footnotesize if also $\nu$  weak log-concave}} &\diagbox{}{}& \makecell{$1-\Theta(\nicefrac{\varepsilon}{R^2})$\\  {\footnotesize if $\supp(\rho),\supp(\nu)\subseteq B_R(0)$}\\
{\footnotesize and $\supp(\rho)$ is convex}} &$1-\Theta( \varepsilon^2\,\nicefrac{\alpha_\rho}{\tau})$\\
\hline
$\rho$ is $(\alpha_\rho,L)$-weak log-concave &\makecell{{\footnotesize $\varepsilon>\varepsilon_0$}\\
{\footnotesize if also $\nu$  weak log-concave}}&\diagbox{}{}&\diagbox{}{}& $1-\Theta(\nicefrac{\varepsilon^2\alpha_\rho^2}{(\alpha_\rho+L)\tau})$
\\
\hline
$\rho$ with light tails & \makecell{\footnotesize{$\varepsilon>\varepsilon_0$}\\ \footnotesize{if also $\nu$ has light tails} }& \makecell{$1-\Theta(e^{-\varepsilon^{-1}})$ \\ \footnotesize{if also $\nu$ has light tails}\\
\footnotesize{weaker notion of light tails}} & \diagbox{}{} & $1-\Theta(\nicefrac{\varepsilon^2}{\tau(1+L^2\,[R^2\vee C^{-\nicefrac2\delta}])})$
\\
\hline
$\supp(\rho)\subseteq B_R(0)$ &\diagbox{} &\makecell{$1-\Theta(e^{-\varepsilon^{-1}})$\\
{\footnotesize if also $\nu$ has light tails}\\
\footnotesize{weaker notion of light tails}}&\makecell{\\ $1-\Theta(\nicefrac{\varepsilon^2}{R^4})$\\
{\footnotesize if also $\supp(\nu)\subseteq B_R(0)$, }\\
{\footnotesize $\supp(\rho)$ is convex}\\
{\footnotesize and $m\leq\log\rho(x)\leq M$}\\\phantom{}}&$1-\Theta(\nicefrac{\varepsilon^2}{\tau R^2})$\\
\hline
\makecell{$\rho(\mathrm{d}x)\propto e^{-|x|^q+\delta|x|^2}$,  $q\geq 2$\\
$\nu(\mathrm{dy})\propto e^{-\min\{|y|^2,|y|^p_p\}}$, $p\in(1,2)$\\ $\nicefrac1p+\nicefrac1q\leq 1$}&\diagbox{} & \diagbox{}&\diagbox{} & $1-\Theta(\nicefrac{\varepsilon}{ K})^\star$\\ 
\hline
\end{tabular}}

\caption{\footnotesize{Comparison of rates for the Euclidean setting with quadratic cost. We always assume $\nu$ to satisfy the transport inequality \ref{eq:TI}. In the table $\Theta(f)$ is a function for which there are universal constants $c,C>0$ such that $c f\leq \Theta(f)\leq C f$. $\star$ : In the last row $\nu$ does not need to satisfy \ref{eq:TI}, since it satisfies the generalized \ref{eq:generalized:TI} with $\tau$ depending on $p,q$.}}
\label{table:comparison}
\end{table}

The main novelties that we introduce in the quadratic setting are the following.
\begin{itemize}
\item We obtain for the first time exponential convergence for all $\varepsilon>0$ under the assumption that $\rho$ is $\alpha_{\rho}$-log-concave, $\nu$ is $\beta_\nu$-log-semiconvex and satisfies a Talagrand inequality. The convergence rate we obtain is $1-\Theta(\frac{\varepsilon}{\tau}\,\sqrt{\nicefrac{\alpha_\rho}{\beta_\nu}})$, which has a sharp dependence on $\varepsilon$ thanks to \cite[Theorem 1.3]{chizat2024sharper}. The only result that covers this setup is obtained in \cite{conforti2023Sinkhorn} but it gives exponential convergence  only for $\varepsilon$ large enough.
    \item We obtain for the first time exponential convergence assuming that $\rho$ is $\alpha_\rho$-log-concave and $\nu$ satisfies \ref{eq:TI}. The convergence rate is $1-\Theta( \varepsilon^2\,\nicefrac{\alpha_\rho}{\tau})$. There are two results we might compare with ours. In the unbounded setting \cite{conforti2023Sinkhorn} proves convergence for $\varepsilon>\varepsilon_0$ (with the noise threshold $\varepsilon_0$ being zero solely in the Gaussian setting). The second one is \cite{chizat2024sharper} where the authors manage to get a rate which in linear in $\varepsilon$ but have to assume that both $\rho$ and $\nu$ have compact support, which we do not. This rate has better dependence in $\varepsilon$ but it deteriorates with the size of the support. 
    Our result shows that indeed stronger log-concavity of $\rho$ improves the convergence rate, addressing an issue raised in \cite[Remark 3.3]{chizat2024sharper}.
\item  For weakly log-concave distributions we obtain for the first time exponential convergence for all $\varepsilon>0$. The only previously known result \cite{conforti2023Sinkhorn} requires $\varepsilon$ to be larger than a threshold value $\varepsilon_0$. Moreover, we show that the rate of convergence is $1-\Theta(\nicefrac{\varepsilon^2\alpha_\rho^2}{L\tau})$ without having to assume neither compactness nor convexity of the marginals' supports.
\item In the light-tails setting, we improve the rate of convergence, from an exponential dependence on $\varepsilon$ in \cite{eckstein2023hilberts}, to the polynomial $1-\Theta(\nicefrac{\varepsilon^2}{\tau(1+L^2\,[R^2\vee C^{-\nicefrac2\delta}])})$. Moreover, we ask that only one of two marginals has tails lighter than Gaussian, whereas \cite{eckstein2023hilberts} requires this property to hold for both marginals. However, one should note that  the result of \cite{eckstein2023hilberts} applies to more general costs and marginals. For example, the light-tail condition considered therein can be expressed as a condition on the growth $U_{\rho}$ rather than its Hessian. Moreover, the cost is not required to have a bounded Hessian.
      \item When dealing with $\rho$ compactly supported, we show a rate $1-\Theta(\nicefrac{\varepsilon^2}{\tau R^2})$ under the assumption that both marginals have compact support and one of them satisfies  \ref{eq:TI}. 
    This should be compared to \cite[Theorem 3.1-(A1)]{chizat2024sharper}. There, the authors get the same dependence in $\varepsilon^2$ assuming that one marginal is compactly supported, and the other one has a uniformly upper- and lower-bounded density on its support, that is taken to be compact and convex. This last assumption implies LSI and is therefore stronger than asking that TI holds, cf.\ \cite[Proposition 5.3]{Bobkov2000FromBrunn}, discussion therein and \cite[Proposition 5.1.6]{bakry2013analysis}. Indeed, if LSI holds, TI also holds.
    As explained in the discussion following Proposition \ref{prop:compact:supp:sink}, the authors also prove that the asymptotic rate of convergence is of order $\varepsilon$.
    \item  In the heavy-tails setting of Proposition \ref{prop_heavy_sink} no rate of convergence for optimal plans were known before, to the best knowledge and understanding. Indeed, the polynomial rates in \cite{ghosal2022nutz} would require both $\rho$ and $\nu$ to be subgaussian, whereas the exponential rates in \cite{eckstein2023hilberts} would require even lighter tails.
\end{itemize}

The proof of all the above propositions is obtained bounding from above the Hessian of Sinkhorn potentials uniformly in $n$ and then invoking Theorem \ref{cor:exp_conv_from_stab}. To control Hessians, we leverage their representation in terms of conditional covariances (cf.~\eqref{eq:hess:hess:cov:psieot}  below). Then, in most cases we proceed to bound covariances by means of functional inequalities and perturbative arguments. In the case of log-concave and weakly log-concave marginals, we argue differently by showing that the map $\Psi^\nu_0(\cdot)$ preserves concavity, and that $\Phi^{\rho}_0(\cdot)$ preserves  convexity and weak convexity. To do so, we rely on Pr\'ekopa--Leindler inequality following  \cite{fathi2019proof, chewi2022entropic} when assuming strong log-concavity, and on the more probabilistic constructions of \cite{conforti2024weak} when assuming weak log-concavity.


\section{Proofs}

\subsection{Preliminaries}

For the proof of \Cref{thm:entropic_stability}, We shall need that the conditional distribution of $\pieot{\nu}$ with respect to the second component can be written as
\be\label{eq:conditional:density}
\begin{aligned}
\pieot{\nu}(\De x|y)&=\exp\biggl(-\frac{c(x,y)+\phieot{\nu}(x)+\psieot{\nu}(y)}{\varepsilon}\biggr) \rho(\De x)\,.\\
\end{aligned}
\ee
Note that in principle the conditional distribution is only defined $\nu$-a.e. Nonetheless, under the semiconcavity assumption  of \Cref{thm:entropic_stability}, one can see that $\psi_\varepsilon^\nu$ is locally bounded from above, and thus we can use the right hand side of \eqref{eq:conditional:density} to extend the definition of the conditional measure to the whole $\cY$. The same considerations apply to Sinkhorn plans in all examples. That is to say
\bes
\begin{aligned}
\pi^{n,n}(\De x|y)&=\exp\biggl(-\frac{c(x,y)+\phieot{n}(x)+\psieot{n}(y)}{\varepsilon}\biggr) \rho(\De x)\,,\\
\pi^{n,n}(\De y|x)&=\exp\biggl(-\frac{c(x,y)+\psieot{n}(y)+\phieot{n}(x)}{\varepsilon}\biggr) \nu(\De y)\,.
\end{aligned}
\ees
are versions of the conditional probability defined everywhere on $\cY$ and $\cX$ respectively.
We will use at several places the following identities along Sinkhorn's algorithm
\be\label{eq:hess:hess:cov:psieot}
\begin{aligned}
\nabla\psieot{n}(y)&=-\int_\cX\nabla_2 c(x,y)\,\pi^{n,n}(\De x|y)\,,\\
\nabla^2\psieot{n}(y)&=-\int_\cX\nabla_2^2 c(x,y)\,\pi^{n,n}(\De x|y)+\varepsilon^{-1} \,\mathrm{Cov}_{X\sim\pi^{n,n}(\cdot|y)}(\nabla_2 c(X,y))\,.
\end{aligned}
\ee
Similarly, whenever also $(\cX,\mathfrak{g})$ is a Riemannian manifold we have
\be\label{eq:hess:hess:cov:phieot}
\begin{aligned}
\nabla\phieot{n+1}(x)&=-\int_\cY\nabla_1 c(x,y)\,\pi^{n,n}(\De y|x)\,,\\
\nabla^2\phieot{n+1}(x)&=-\int_\cY\nabla_1^2 c(x,y)\,\pi^{n,n}(\De y|x)+\varepsilon^{-1} \,\mathrm{Cov}_{Y\sim\pi^{n,n}(\cdot|x)}(\nabla_1 c(x,Y))\,.
\end{aligned}
\ee 
These estimates actually hold true also for the limit entropic potentials in a weak sense and have been already extensively studied in the EOT literature (see \cite[Lemma 6]{chewi2022entropic} and \cite[Appendix A]{conforti2023projected}).

Finally, let us point out here that the identities \eqref{eq:hess:hess:cov:phieot} are not employed in the proofs of Theorem \ref{thm:entropic_stability}, Theorem \ref{thm:entropic_stability:gen},  Theorem \ref{cor:exp_conv_from_stab} and in Theorem \ref{thm:generalized:sink}, where $\cX$ can be taken to be Polish. In particular, solely \Cref{thm:entropic_stability} and \Cref{thm:entropic_stability:gen} require the validity of the first identity appearing in~\eqref{eq:hess:hess:cov:psieot} for the limit entropic potential $\psieot{\nu}$ (or equivalently the validity of \eqref{eq:diff:assumpotion}), which is guaranteed by \Cref{ass:differentiability}. 
Concerning Theorem \ref{cor:exp_conv_from_stab} and \Cref{thm:generalized:sink}, our proof technique requires the validity of the first equation in~\eqref{eq:hess:hess:cov:psieot}  and hence the validity of \eqref{eq:diff:assumpotion} along Sinkhorn iterates $\psieot{n}$. In order to guarantee the validity of the differentiation under the integral sign is enough starting Sinkhorn's algorithm with $\phieot{0}$ smooth enough and requiring $c(x,\cdot)$ to be $\cC^1(\cY)$.

Furthermore, the second identity in~\eqref{eq:hess:hess:cov:psieot} as well as the identities in~\eqref{eq:hess:hess:cov:phieot} are sometimes used to study the semiconcavity of potentials in  the examples we discuss at~\cref{sec:esempi}.
In all these cases, $\cX$ is either the Euclidean space or a smooth Riemannian manifold, the cost is either $\cC^2$ with bounded Hessian or Lipschitz, and the assumptions on the marginals provide enough regularity to argue as in \cite[Proposition 4.4 and Lemma 4.5]{conforti2023projected} (combined with \Cref{lem_der_under_integral_sign}). In there, the authors show that $\phieot{\nu}$ and $\psieot{\nu}$ are $\cC^1$ with H\"older continuous first derivative, and that the gradient identities hold true whereas the Hessian identities hold in a weak sense. An independent proof of both identities in a strong sense in the quadratic setting can be found in \cite[Lemma 6]{chewi2022entropic}.

\subsection{Proof of the main results}
In this section we provide the proofs of \Cref{thm:entropic_stability}, \Cref{thm:entropic_stability:gen}, \Cref{cor:exp_conv_from_stab} and \Cref{thm:generalized:sink}. The key estimate required for \Cref{thm:entropic_stability} is contained in \Cref{lem:entropy_between_conditionals}, which bounds the relative entropy of the conditional distribution of the entropic plan at two different points by their distance squared.

\begin{lemma}\label{lem:entropy_between_conditionals}
Under the assumptions of Theorem \ref{thm:entropic_stability}, let $\pieot{\nu}$ be the entropic plan between $\rho$ and $\nu$ and let $\pieot{\nu}(\cdot|y)$ denote the conditional distribution (conditioned on the second variable being equal to $y$). Then for all $y,z\in\cY$ we have	
\bes
\scrH(\pieot{\nu}(\cdot|y)|\pieot{\nu}(\cdot|z)) \leq \frac{\Lambda}{2\varepsilon}\,\sfd^2(y,z)\,.
\ees
\end{lemma}
\begin{proof}
Let $y,z$ be given and define the function 
\be
g(x,\bar y) = c(x,\bar y)+\psieot{\nu}(\bar y)
\ee
By assumption, $\bar y\mapsto g(x,\bar y)$ is $\Lambda$-semiconcave uniformly in $x\in\supp(\rho)$. From the representation of conditional distributions \eqref{eq:conditional:density} we immediately get
\be\label{eq:conditional_entropy_bound}
\begin{aligned}
 \varepsilon\,\scrH(\pieot{\nu}(\cdot|y)|\pieot{\nu}(\cdot|z)))& = \int_{\cX}g(x,z)-g(x,y)\,\pieot{\nu}(\De x|y)\\
&\leq \left\langle\dot{\gamma}_0 ,\, \int_{\cX} \nabla_2g(x,y)\,\pieot{\nu}(\De x|y) \right\rangle_{\!\!\mathfrak{g}}+ \frac{\Lambda}{2} \,\sfd^2(y,z)\,,
\end{aligned}
\ee
where $(\gamma_t)_{t\in[01]}$ is a geodesic from $y$ to $z$. Next, from~\eqref{eq:diff:assumpotion} we see that
\bes
\nabla_2g(x,y) = \nabla_2c(x,y)-\int_{\cX} \nabla_2c(\bar x,y)\, \pieot{\nu}(\De \bar x|y)\,,
\ees
from which it follows that 
\bes
\int_{\cX} \nabla_2g(x,y)\,\pieot{\nu}(\De x|y)=0\,.
\ees
Using this identity in \eqref{eq:conditional_entropy_bound} gives the desired result.
\end{proof}
\begin{proof}[Proof of \Cref{thm:entropic_stability}]
We assume without loss of generality that  $\scrH(\mu|\nu),\,\Lambda,\,\bfW_2^2(\mu,\nu)$ are all finite, otherwise there is nothing to prove, in particular $\mu$ is absolutely continuous with respect to $\nu$. First, we recall that $\pieot{\mu}$ can be seen as a Schr\"odinger optimal plan w.r.t.\ the reference measure $\pieot{\nu}$ for the Schr\"odinger problem, that is the (unique minimizer) for
\be\label{eq:SB_relative} 
\scrH(\pieot{\mu}|\pieot{\nu})= \min_{\pi\in\Pi(\rho,\mu)}\scrH(\pi|\pieot{\nu})\,.
\ee
This directly follows from \cite[Theorem 2.1.b]{Marcel:notes} after noticing that 
\bes
\frac{\De \pieot{\mu}}{\De \pieot\nu} = \exp\biggl(\frac{(\phieot{\nu}-\phieot{\mu})\oplus(\psieot{\nu}-\psieot{\mu})}{\varepsilon}\biggr)\,\frac{\De\mu}{\De\nu}\,,\quad\rho\otimes\nu\text{-a.s.}
\ees 
and hence also $\pieot{\nu}$-a.s.\ (since $\scrH(\pieot{\nu}|\rho\otimes\nu)<\infty$). 
We now proceed to bound $\scrH(\pieot{\mu}|\pieot{\nu})$ exhibiting a suitable admissible plan in~\eqref{eq:SB_relative}.  To this aim, fix a coupling  $\tau\in\Pi(\mu,\nu)$ between $\mu$ and $\nu$ and let us consider the coupling
\bes
\pi(\De x,\De y)\coloneqq \mu(\De y) \int_\cY \pieot{\nu}(\De x|z) \tau(\De z|y )\,.
\ees
Notice that $\pi\in\Pi(\rho,\mu)$ and therefore from \eqref{eq:SB_relative} it follows $\scrH(\pieot{\mu}|\pieot{\nu})\leq \scrH(\pi|\pieot{\nu})$. From the disintegration property of relative entropy (see for instance \cite[Appendix A]{LeoSch}) and from its convexity we may deduce that
\bes
	\begin{aligned}
		\scrH(\pieot{\mu}|\pieot{\nu})&\, \leq\scrH(\pi|\pieot{\nu}) \leq \scrH(\mu|\nu) + \int_{\cY} \scrH(\pi(\cdot|y)|\pieot{\nu}(\cdot|y))\mu(\De y)\\
		&\,\leq \scrH(\mu|\nu) + \int_\cY\int_\cY \scrH(\pieot{\nu}(\cdot|z)|\pieot{\nu}(\cdot|y))\tau(\De z|y)\mu(\De y)\\
		&\overset{\mathrm{\Cref{lem:entropy_between_conditionals}}}{\leq}  \scrH(\mu|\nu)  +\frac{\Lambda}{2\varepsilon}\,\int_\cY \sfd^2(y,z)\,\tau(\De y,\, \De z)\,.
	\end{aligned}
\ees
The desired conclusion follows by optimizing over $\tau\in\Pi(\mu,\nu)$.
\end{proof}

The proof of \Cref{thm:entropic_stability:gen} can be obtained in the same way and for this reason is omitted. Here we solely mention that the key estimate required for establishing \Cref{thm:entropic_stability:gen} is showing that for all $y,z\in\cY$ we have	
\bes
\scrH(\pieot{\nu}(\cdot|y)|\pieot{\nu}(\cdot|z)) \leq \frac{\Lambda}{2\varepsilon}\,\omega(y,z)\,,
\ees
which can be obtained as done in \Cref{lem:entropy_between_conditionals} for the case $\omega=\sfd^2$.

\begin{proof}[Proof of Theorem \ref{cor:exp_conv_from_stab}]
We start recalling the known identity
\be\label{entropy_difference_along_sink}
\begin{aligned}
\scrH(\pieot{\nu}|\pi^{n+1,n})-\scrH(\pieot{\nu}|\pi^{n,n-1}) 
&=\frac{1}{\varepsilon}\int (\phieot{n+1}-\phieot{n})\oplus(\psieot{n}-\psieot{n-1})\,\De\pieot{\nu}\\
&\stackrel{\eqref{eq:wrong_marginals}}{=}-(\scrH(\rho|\rho^{n,n})+\scrH(\nu|\nu^{n,n-1}))\,.
\end{aligned}
\ee
\noindent\textbf{Proof of \ref{item:exp_conv_from_stab_i}.} Let $n\geq N-1.$ From Theorem \ref{thm:entropic_stability}, applied to the pairs of marginals $(\rho,\nu)$ and $(\rho,\nu^{n+1,n })$, we obtain 
\bes
\begin{split}
\scrH(\pieot{\nu}|\pi^{n+1,n}) &\leq\,\scrH(\nu|\nu^{n+1,n})+\frac{\Lambda}{2\varepsilon} \bfW^2_2(\nu^{n+1,n},\nu)\\
&\stackrel{\eqref{eq:TI}}{\leq}\, \scrH(\nu|\nu^{n+1,n})+\frac{\tau\Lambda}{\varepsilon}\scrH(\nu^{n+1,n}|\nu)\\
&\leq\, \max\{1,\,\nicefrac{\tau\Lambda}{\varepsilon}\}\big( \scrH(\nu^{n+1,n}|\nu)+\scrH(\nu|\nu^{n+1,n})\big)\,.
\end{split}
\ees
Invoking Sinkhorn's monotonicity inequalities \cite[Proposition 6.10]{Marcel:notes}
\bes
\scrH(\nu^{n+1,n}|\nu) \leq \scrH(\rho|\rho^{n,n})\,, \quad \scrH(\nu|\nu^{n+1,n})\leq \scrH(\nu|\nu^{n,n-1})\,,
\ees
we arrive at the following bound
\bes
\scrH(\pieot{\nu}|\pi^{n+1,n}) \leq \max\{1,\,\nicefrac{\tau\Lambda}{\varepsilon}\}(\scrH(\rho|\rho^{n,n})+\scrH(\nu|\nu^{n,n-1}))\,.
\ees
Using this result in \eqref{entropy_difference_along_sink} gives 
\bes
\scrH(\pieot{\nu}|\pi^{n+1,n})-\scrH(\pieot{\nu}|\pi^{n,n-1}) \leq - \min\{1,\,\nicefrac{\varepsilon}{\tau\Lambda}\}\scrH(\pieot{\nu}|\pi^{n+1,n})\,.
\ees
We thus obtain from a simple recursion that for all $n\geq N-1$
\bes
\scrH(\pieot{\nu}|\pi^{n+1,n}) \leq (1+\min\{1,\,\nicefrac{\varepsilon}{\tau\Lambda}\})^{-(n-N+2)}\scrH(\pieot{\nu}|\pi^{N-1,N-2})\,.
\ees
From the monotonicity bounds (\cite[Proposition 6.5]{Marcel:notes})
\be\label{eq:sink_mon}
\scrH(\pieot{\nu}|\pi^{m+1,m+1})\leq \scrH(\pieot{\nu}|\pi^{m+1,m})\leq \scrH(\pieot{\nu}|\pi^{m,m})\,, \quad \forall m\geq 0\,,
\ee
we obtain  
\bes
\scrH(\pieot{\nu}|\pi^{n+1,n}) \leq (1+\min\{1,\,\nicefrac{\varepsilon}{\tau\Lambda}\})^{-(n-N+2)}\,\scrH(\pieot{\nu}|\pi^{1,0})\,, \quad \forall n\geq N-1\,.
\ees
The desired conclusion follows from the bounds
\be
\scrH(\pieot{\nu}|\pi^{1,0})\leq\scrH(\pieot{\nu}|\pi^{0,0}), \quad \scrH(\pieot{\nu}|\pi^{n+1,n+1})\leq \scrH(\pieot{\nu}|\pi^{n+1,n})\,,
\ee
which are a consequence of \eqref{eq:sink_mon}.
\vspace{\baselineskip}

\noindent\textbf{Proof of \ref{item:exp_conv_from_stab_ii}} 
Owing to~\Cref{thm:entropic_stability} applied to the pairs of marginals $(\rho,\nu)$ and $(\rho,\nu^{n,n-1})$ we see that 
\be\label{eq:stability_bound_in_sink}
\scrH(\pieot{\nu}|\pi^{n,n-1}) \leq \scrH(\nu|\nu^{n,n-1})+\frac{\Lambda}{2\varepsilon} \bfW^2_2(\nu^{n-1,n},\nu)\,,
\ee
and since $\nu^{n,n-1}$ satisfies \ref{eq:TI} for $n\geq N$ we conclude that
\bes
\scrH(\nu|\nu^{n,n-1}) \geq \big(1+\nicefrac{\tau\Lambda}{\varepsilon}\big)^{-1}\scrH(\pieot{\nu}|\pi^{n,n-1})\,.
\ees
Using this bound in  \eqref{entropy_difference_along_sink} then gives
\bes
\scrH(\pieot{\nu}|\pi^{n+1,n})\leq\big(1+\nicefrac{\varepsilon}{\tau\Lambda}\big)^{-1}\,\scrH(\pieot{\nu}|\pi^{n,n-1})\,.
\ees
The desired conclusion follows from a simple recursion and the monotonicity bounds \eqref{eq:sink_mon}.
\end{proof}

Similarly, from \Cref{thm:entropic_stability:gen} we can establish the exponential convergence of Sinkhorn's algorithm as stated in \Cref{thm:generalized:sink}.
\begin{proof}[Proof of \Cref{thm:generalized:sink}]
The proof of Parts \ref{item:exp_conv_from_stab_i:gen} and \ref{item:exp_conv_from_stab_ii:gen} runs as shown for  Parts \ref{item:exp_conv_from_stab_i} and \ref{item:exp_conv_from_stab_ii} of \Cref{cor:exp_conv_from_stab}, this time by relying on \Cref{thm:entropic_stability:gen} and on the generalized transport inequality~\ref{eq:generalized:TI}, instead of~\ref{eq:TI}.
\end{proof}

\medskip

Before moving to the computation of convergence rates, we recall that a probability measure $\rho$ satisfies a Poincar\'e inequality \ref{def:poincare:ineq} with constant $C_\rho>0$ if for any $f\in W^{1,2}(\rho)$ it holds
\be\label{def:poincare:ineq}\tag{PI($C_\rho$)}
\mathrm{Var}_{\rho}(f)\coloneqq \bbE_{\rho}[f^2(X)]-\bbE_{\rho}[f(X)]^2\leq C_\rho\, \int_\cX|\nabla f|^2\De \rho\,.
\ee
Moreover, let us recall here that any $\alpha$-log-concave measure satisfies the Talagrand transport inequality $\mathrm{TI}(\alpha^{-1})$ and the Poincar\'e inequality $\mathrm{PI}(\alpha^{-1})$ (see \cite[Corollaries 4.8.2 and 9.3.2]{bakry2013analysis}).

\subsection{Proof of Proposition \ref{prop:anisotropic:cost}}

Clearly we have $\nabla_2^2 c(x,y)=\Sigma$ and hence 
\be\label{anisotropic:covariance}
\nabla_2^2(c(x,y)+\psieot{n}(y))=\Sigma+\nabla^2\psieot{n}(y)\overset{\eqref{eq:hess:hess:cov:psieot}}{=}\varepsilon^{-1} \,\mathrm{Cov}_{X\sim\pi^{n,n}(\cdot|y)}(\Sigma(y-X))\,.
\ee
Therefore in this section, in order to study the semiconcavity of $y\mapsto c(x,y)+\psieot{n}(y)$, it is enough to control the conditional covariance matrices.

\subsubsection{Log-concavity of $\rho$} We start with the proof of (ii).
From \eqref{eq:hess:hess:cov:phieot} (applied along Sinkhorn's algorithm) it follows $\nabla^2\phieot{n}(x) \succeq -\Sigma$ for any $n\in\N$, which combined with \eqref{densities:along:sink} further implies that for all $y$
\bes
\nabla^2_1\big(-\log\pi^{n,n}(x|y)\big) \succeq \nabla^2U_\rho(x)\succeq \alpha_\rho\,,
\ees
where we wrote $\pi^{n,n}(x|y)$ for the density of $\pi^{n,n}(\De x|y)$ with respect to  the Lebesgue measure.
This guarantees that, uniformly in $y\in\bbRD$ and $n\in\N$, the conditional measure $\pi^{n,n}(\De x|y)$ satisfies the Poincar\'e inequality $\mathrm{PI}(\alpha_\rho^{-1})$ (cf. \cite[Corollaries 4.8.2]{bakry2013analysis}).

Then, from \eqref{anisotropic:covariance} we deduce that for any unit vector $v$ it holds
\bes
\varepsilon\,\langle v,\,\nabla_2^2(c(x,y)+\psieot{n}(y))v\rangle =\mathrm{Var}_{X\sim\pi^{n,n}(\cdot|y)}(\langle v,\,\Sigma X \rangle)\leq \alpha_\rho^{-1}\,\|\Sigma v\|^2\leq \alpha_\rho^{-1}\,\|\Sigma\|_2^2 \,.
\ees
This means that for any $n\in\N$ and uniformly in $x\in\supp(\rho)$ the map  $y\mapsto c(x,y)+\psieot{n}(y)$ is $\Lambda$-semiconcave on $\supp(\nu)$ with $\Lambda= (\varepsilon\alpha_\rho)^{-1}\|\Sigma\|_2^2$. This combined with \Cref{cor:exp_conv_from_stab}-\ref{item:exp_conv_from_stab_i} proves the exponential proves convergence of Sinkhorn's algorithm. If $\varepsilon\leq\|\Sigma\|_2\,\sqrt{\nicefrac{\tau}{\alpha_{\rho}}}$, then $\varepsilon\leq\tau\Lambda$  and the rate takes the form \eqref{eq:rate_log_conc}.

\subsubsection{Log-concavity of $\rho$ and log-semiconvexity of $\nu$}

We now discuss (i). Strengthening our assumption on the marginals, we can improve on the convergence rate. 
As a starter, we show the following convexity/concavity result that generalizes what is already known for the Euclidean quadratic cost in \cite{chewi2022entropic} and \cite[Theorem 10]{conforti2023Sinkhorn}.

\begin{lemma}\label{lemma:anisotropic:caff}
    Assume $\Sigma\succ 0$,  $\nabla^2U_\rho\succeq \alpha_\rho$ and $\nabla^2U_\nu\preceq\beta_\nu$ for some $\alpha_\rho>0$ and $\beta_\nu\in (0,+\infty)$. If $\nabla^2\phieot{0}\succeq -\Sigma+A_0\,\Sigma$ for some matrix $A_0\succeq 0$ commuting with $\Sigma$, then for any $n\in\N$ we have
    \be\label{eq:iterated:bound:caff}
    \nabla^2\phieot{n}\succeq -\Sigma +A_n\,\Sigma\quad\text{ and }\quad\nabla^2\psieot{n}\preceq -\Sigma+ B_n\,\Sigma\,,
    \ee
    where $(A_n)_{n\in\N}$ and $(B_n)_{n\in\N}$ are two sequences of positive semidefinite symmetric matrices iteratively defined via
    \be\label{def:iterated:caff}
    \begin{cases}
        B_{n}=\Sigma (A_n\Sigma+\varepsilon\alpha_\rho)^{-1}\,,\\
        A_{n+1}=        \Sigma(B_n\Sigma+\varepsilon\beta_\nu)^{-1}\,,
    \end{cases}
    \ee
    and converging to 
\be\label{Ainfty}
\begin{aligned}
A_\infty =&\,-\frac{\varepsilon\alpha_\rho}2\Sigma^{-1}+\biggl(\frac{\varepsilon^2\alpha_\rho^2}{4}\Sigma^{-2}+\frac{\alpha_\rho}{\beta_\nu}\biggr)^{\nicefrac12}\,,\\
B_\infty=&\,-\frac{\varepsilon\beta_\nu}2\Sigma^{-1}+\biggl(\frac{\varepsilon^2\beta_\nu^2}{4}\Sigma^{-2}+\frac{\beta_\nu}{\alpha_\rho}\biggr)^{\nicefrac12}\,.
\end{aligned}
\ee
\end{lemma}

Let us briefly notice before the proof that usually Sinkhorn's algorithm is initialized at $\phieot{0}=0$, for which the assumption $\nabla^2\phieot{0}\succeq -\Sigma$ holds with $A_0=0$. Nonetheless, as for any $n\geq 1$, $\nabla^2\phieot{n}\succeq-\Sigma$ by \eqref{eq:hess:hess:cov:phieot}, one can get a similar conclusion for any initialization and all $n \geq 1$ by replacing \eqref{eq:iterated:bound:caff} with $\nabla^2\phieot{n}\succeq -\Sigma +A_{n-1}\,\Sigma$ and $\nabla^2\psieot{n}\preceq -\Sigma+ B_{n-1}\,\Sigma$, where $(A_n)_{n\in\N}$ and $(B_n)_{n\in\N}$ satisfy \eqref{def:iterated:caff}.

\begin{proof}
    The idea of the proof is to mimic the iterative proof given in \cite{conforti2024weak, conforti2023Sinkhorn} combined with covariance bounds as in \cite{chewi2022entropic}. We will proceed by induction. The base case holds true with $A_0\succeq 0$, by assumption. Next, let us assume that $\nabla^2\phieot{n}\succeq -\Sigma+A_n\Sigma$ for some $A_n\succeq 0$. This  implies that the conditional measure $\pi^{n,n}(\cdot|y)$ is log-concave since
    \bes
    \nabla^2_1(-\log\pi^{n,n}(x|y))=\varepsilon^{-1}(\Sigma+\nabla^2\phieot{n}(x))+\nabla^2U_\rho(x)\succeq \varepsilon^{-1}A_n\Sigma+\alpha_\rho\,.
    \ees
    Then, from the Brascamp–-Lieb inequality (see for instance \cite[Lemma 2]{chewi2022entropic}), we can deduce that
    \bes
    \begin{aligned}
\nabla^2\psieot{n}(y)\overset{\eqref{eq:hess:hess:cov:psieot}}{=}-\Sigma+\varepsilon^{-1}\mathrm{Cov}_{X\sim\pi^{n,n}(\cdot|y)}(\Sigma X)=-\Sigma+\varepsilon^{-1}\Sigma\,\mathrm{Cov}_{X\sim\pi^{n,n}(\cdot|y)}(X)\Sigma\\
\preceq -\Sigma +\Sigma (A_n\Sigma+\varepsilon\alpha_\rho)^{-1}\Sigma\,,
    \end{aligned}
    \ees
    which proves $\nabla^2\psieot{n}\preceq -\Sigma +B_n\Sigma$. This further implies that the conditional probability measure $\pi^{n+1,n}(\cdot|x)$ is log-convex uniformly in $x\in\bbRD$ since
    \bes
\nabla^2_2(-\log\pi^{n+1,n}(y|x))=\varepsilon^{-1}(\Sigma+\nabla^2\psieot{n}(y))+\nabla^2U_\nu(y)\preceq   
\varepsilon^{-1}\,B_n\Sigma+\beta_\nu\,.
    \ees
    This, combined with the Cram\'er--Rao inequality (see for instance \cite[Lemma 2]{chewi2022entropic}) gives uniformly in $x\in\bbRD$
    \bes\begin{aligned}
\nabla^2\phieot{n+1}(x)\overset{\eqref{eq:hess:hess:cov:phieot}}{=}-\Sigma+\varepsilon^{-1}\mathrm{Cov}_{Y\sim\pi^{n+1,n}(\cdot|x)}(\Sigma Y)=-\Sigma+\varepsilon^{-1}\,\Sigma\,\mathrm{Cov}_{Y\sim\pi^{n+1,n}(\cdot|x)}( Y)\Sigma \\
\succeq -\Sigma+(B_n\Sigma+\varepsilon\beta_\nu)^{-1}\Sigma\,,
   \end{aligned} \ees
    which proves $\nabla^2\phieot{n+1}\succeq-\Sigma+A_{n+1}\Sigma$. This shows the validity of \eqref{eq:iterated:bound:caff} with the sequences defined at \eqref{def:iterated:caff}. Moreover, the same induction argument shows that whenever $\Sigma\succ0$ we are guaranteed for any $n\in\N$ that $A_n\succ0$, $B_n\succ0$ and both matrices commute with $\Sigma$. The rest of the proof is devoted to showing that the sequence $(A_n,B_n)_{n\in\N}$ converge to the fixed point of \eqref{def:iterated:caff}, so that \eqref{Ainfty} follows from the validity of \eqref{def:iterated:caff} in such limit points.
In view of that, let us preliminary notice that since $A_0$ and $\Sigma$ commute, both are jointly diagonalizable over a basis $\{v_1,\dots,\,v_d\}$ of $\bbRD$. Moreover since $\Sigma$ is non-singular, notice that \eqref{def:iterated:caff} can be rewritten as
\be\label{eq:2step:Bn}
B_n=(A_n+\varepsilon\alpha_\rho\Sigma^{-1})^{-1}\,\quad\text{ and }\quad A_{n+1}=(B_n+\varepsilon\beta_\nu\Sigma^{-1})^{-1}\,.
\ee
From this we immediately deduce that with respect to the same basis $\{v_1,\dots,\,v_d\}$ the matrices $A_n$ and $B_n$ are diagonal as well for each $n\geq 0$ and their eigenvalues (respectively $\{a_n^1,\dots,\,a_n^d\}$ and $\{b_n^1,\dots,\,b_n^d\}$) satisfy
\be\label{eq:iterate:diagonali}
b_n^k=(a_n^k+\varepsilon \alpha_\rho\Sigma_{kk}^{-1})^{-1}\,\quad\text{ and }\quad a_{n+1}^k=(b_n^k+\varepsilon\beta_\nu\Sigma_{kk}^{-1})^{-1}\,\quad\forall\,k=1,\dots,\,d\,,
\ee
where $\Sigma_{kk}>0$ denotes the $k^{th}$ diagonal entry of $\Sigma$ over the basis $\{v_1,\cdots,\,v_d\}$. For each $k=1,\dots,\,d$ we will show that $(a_n^k)_{n\in\N}$  is and $(b_n^k)_{n\in\N}$ are converging to positive limits.

Firstly, assume that $a^k_1<a^k_0$. Then $(a_n^k)_{n\in\N}$  is monotone decreasing. Indeed, if  we assume that $a_n^k>a_{n-1}^k$ then we have
    \bes\begin{aligned}
 b_{n}^k-b_{n-1}^k=\frac1{a_n^k+\frac{\varepsilon\alpha_\rho}{\Sigma^{-1}_{kk}}}-\frac1{a_{n-1}^k+\frac{\varepsilon\alpha_\rho}{\Sigma^{-1}_{kk}}}<0\,,
        \end{aligned}\ees
which further implies that    
\bes
a_{n+1}^k-a_n^k=\frac1{b_n^k+\frac{\varepsilon\beta_\nu}{\Sigma_{kk}}}-\frac1{b_{n-1}^k+\frac{\varepsilon\beta_\nu}{\Sigma_{kk}}}>0\,.
\ees
Therefore by induction we have shown that if $a^k_1<a^k_0$ then $(a_n^k)_{n\in\N}$  is monotone decreasing. On the other hand,  if $a^k_1\geq a^k_0$, then  we can prove that $(a_n^k)_{n\in\N}$  is monotone non-decreasing and $(b_n^k)_{n\in\N}$ is monotone non-increasing, by the same argument. Hence either one between $(a_n^k)_{n\in\N}$  and $(b_n^k)_{n\in\N}$ is monotone non-increasing and lower-bounded (since for any $n\in\N$ that $A_n\succ0$, $B_n\succ0$ are positive definite). From this we may thus deduce the convergence of either one between $(a_n^k)_{n\in\N}$  and $(b_n^k)_{n\in\N}$, which implies the convergence of the other one via~\eqref{eq:iterate:diagonali}. In conclusion we have shown that for each  $k=1,\dots,\,d$ the sequences $(a_n^k)_{n\in\N}$  and $(b_n^k)_{n\in\N}$ converge to some positive limit points.

In particular, this shows that the sequence $(B_n)_{n\in\N}$ converges to a positive semidefinite matrix $B_\infty$. Moreover, from~\eqref{eq:2step:Bn} we immediately see that the limit-matrix solves
    \be\label{eq:Binfty_char}
B_\infty=[(B_\infty+\varepsilon\beta_\nu\Sigma^{-1})^{-1}+\varepsilon\alpha_\rho\Sigma^{-1}]^{-1}\,,
    \ee
Observe that, since $B_\infty$ is positive semidefinite, the right hand side in the above identity dominates a positive definite matrix, whence $B_\infty$ is positive definite and invertible. Moreover, since $B_\infty$ and $\Sigma$ commute, we can rewrite \eqref{eq:Binfty_char} as
\bes\begin{aligned}
0=&\,\alpha_\rho\,B_\infty^2+\varepsilon\alpha_\rho\beta_\nu B_\infty\Sigma^{-1}-\beta_\nu\\
=&\,\alpha_\rho\,\biggl(B_\infty+\frac{\varepsilon\beta_\nu}2\Sigma^{-1}+\biggl(\frac{\varepsilon^2\beta_\nu^2}{4}\Sigma^{-2}+\frac{\beta_\nu}{\alpha_\rho}\biggr)^{\nicefrac12}\biggr)\biggl(B_\infty+\frac{\varepsilon\beta_\nu}2\Sigma^{-1}- \biggl(\frac{\varepsilon^2\beta_\nu^2}{4}\Sigma^{-2}+\frac{\beta_\nu}{\alpha_\rho}\biggr)^{\nicefrac12}\biggr)\,,
\end{aligned}\ees
where the square root has to be understood as the square root of a positive semidefinite matrix. Since $B_\infty$ ought to be positive definite we conclude that
\bes
B_\infty=-\frac{\varepsilon\beta_\nu}2\Sigma^{-1}+\biggl(\frac{\varepsilon^2\beta_\nu^2}{4}\Sigma^{-2}+\frac{\beta_\nu}{\alpha_\rho}\biggr)^{\nicefrac12}\,.
\ees
The convergence of $(B_n)_{n\in\N}$ to $B_\infty$ implies the convergence of $(A_n)_{n\in\N}$ towards the limit matrix
\bes
A_\infty=(B_\infty+\varepsilon\beta_\nu\Sigma^{-1})^{-1}=-\frac{\varepsilon\alpha_\rho}2\Sigma^{-1}+\biggl(\frac{\varepsilon^2\alpha_\rho^2}{4}\Sigma^{-2}+\frac{\alpha_\rho}{\beta_\nu}\biggr)^{\nicefrac12}\,.
\ees
\end{proof}

If in the previous lemma we consider as Sinkhorn's iterates the constant iterates $\phieot{n}=\phieot{\nu}$ and $\psieot{n}=\psieot{\nu}$ (\ie~we start the algorithm already in the fixed point $(\phieot{\nu},\psieot{\nu})$), then we immediately deduce the following.

\begin{corollary}
    When the two marginals satisfy \eqref{conditio:caff} and $\Sigma\succ 0$, we have
    \bes
    \begin{aligned}
    \nabla^2\phieot{\nu}&\succeq-\Sigma -\frac{\varepsilon\alpha_\rho}2+\biggl(\frac{\varepsilon^2\alpha_\rho^2}{4}+\frac{\alpha_\rho}{\beta_\nu}\,\Sigma^2\biggr)^{\nicefrac12}\,,\\
    \nabla^2\psieot{\nu}&\preceq -\Sigma- \frac{\varepsilon\beta_\nu}2+\biggl(\frac{\varepsilon^2\beta_\nu^2}{4}+\frac{\beta_\nu}{\alpha_\rho}\,\Sigma^2\biggr)^{\nicefrac12}\,.
    \end{aligned}
    \ees
\end{corollary}

\medskip

Let us also highlight here a regularizing property that can be immediately deduced from \Cref{lemma:anisotropic:caff}.

\begin{corollary}\label{cor:warm:start:condition}
When the two marginals satisfy \eqref{conditio:caff} and $\Sigma\succ 0$, if $A_0\succeq A_\infty$ (defined at~\eqref{Ainfty}) then $B_n\preceq B_\infty$ for any $n\in\N$. Equivalently, if we start Sinkhorn's algorithm in $\phieot{0}$ such that 
 \be\label{eq:warm:start:condition}
 \nabla^2\phieot{0}\succeq -\Sigma -\frac{\varepsilon\alpha_\rho}2+\biggl(\frac{\varepsilon^2\alpha_\rho^2}{4}+\frac{\alpha_\rho}{\beta_\nu}\,\Sigma^2\biggr)^{\nicefrac12}\,,
 \ee
 then for any $n\in\N$
 \be\label{eq:bound:warm:iterated}
 \nabla^2\psieot{n}\preceq -\Sigma- \frac{\varepsilon\beta_\nu}2+\biggl(\frac{\varepsilon^2\beta_\nu^2}{4}+\frac{\beta_\nu}{\alpha_\rho}\,\Sigma^2\biggr)^{\nicefrac12}\,.
 \ee
\end{corollary}

\begin{proof}
Let us start by showing that if $A_0\succeq A_\infty$, then $B_n\preceq B_\infty$. In view of this, it suffices to show by induction that $A_n-A_\infty\succeq 0$ and $B_n-B_\infty\preceq 0$. The base case is the standing assumption. Next, assume $A_n-A_\infty\succeq 0$. From \eqref{def:iterated:caff} (and by recalling that $(A_\infty,B_\infty)$ is a fixed point solution of the latter) we immediately deduce that
\bes
B_n-B_\infty= (A_n+\varepsilon\alpha_\rho\Sigma^{-1})^{-1}-(A_\infty+\varepsilon\alpha_\rho\Sigma^{-1})^{-1}\preceq 0\,,
\ees
since our inductive step implies $(A_n+\varepsilon\alpha_\rho\Sigma^{-1})-(A_\infty+\varepsilon\alpha_\rho\Sigma^{-1})\succeq 0$. This shows $B_n-B_\infty\preceq 0$, which further implies (same reasoning from \eqref{def:iterated:caff}) $A_{n+1}-A_\infty\succeq 0$, which concludes our inductive proof.
    
Finally, if \eqref{eq:warm:start:condition} holds, then we can take $A_0\succeq A_\infty$ in \Cref{lemma:anisotropic:caff}, hence the above discussion yields $B_n\preceq B_\infty$ which implies \eqref{eq:bound:warm:iterated}.
\end{proof}

We are now able to conclude our analysis of the semiconcavity of $y\mapsto c(x,y)+\psieot{n}(y)$. 
For a generic initial condition, we know by Lemma \ref{lemma:anisotropic:caff} that $\nabla\psieot{n} \preceq -\Sigma+B_n\Sigma$. Therefore for any $n\in\N$ the map $y\mapsto c(x,y)+\psieot{n}(y)$ is $\norm{B_n\,\Sigma}_2$-semiconcave. Since $B_n$ converges to $B_\infty$  we have
    \bes
    \begin{aligned}
    \lim_{n\to \infty} \norm{B_n\,\Sigma}_2=\norm{B_\infty\Sigma}_2=\norm{-\frac{\varepsilon\beta_\nu}2+\biggl(\frac{\varepsilon^2\beta_\nu^2}{4}+\frac{\beta_\nu}{\alpha_\rho}\Sigma^2\biggr)^{\nicefrac12}}_2\\
    =-\frac{\varepsilon\beta_\nu}2+\biggl(\frac{\varepsilon^2\beta_\nu^2}{4}+\frac{\beta_\nu}{\alpha_\rho}\norm{\Sigma}_2^2\biggr)^{\nicefrac12}
    \leq \sqrt{\nicefrac{\beta_\nu}{\alpha_\rho}}\,\norm{\Sigma}_2\,,
    \end{aligned}
    \ees
and therefore \Cref{cor:exp_conv_from_stab}-\ref{item:exp_conv_from_stab_i} implies the exponential convergence of Sinkhorn's algorithm and for $\varepsilon\leq \tau \,\sqrt{\nicefrac{\beta_\nu}{\alpha_\rho}}\,\norm{\Sigma}_2$ the asymptotic rate is
    \be\label{eq:rate:caff}
    1-\frac{\varepsilon}{\varepsilon+\tau\|\Sigma\|_2(\nicefrac{\beta_\nu}{\alpha_\rho})^{\nicefrac12}}\,.
    \ee
Finally, if we start Sinkhorn's algorithm with $\phieot{0}$ satisfying \eqref{eq:warm:start:condition}, then \Cref{cor:warm:start:condition} implies the validity of \eqref{eq:bound:warm:iterated} for all $n\in\mathbb{N}$. Thus, $y\mapsto c(x,y)+\psieot{n}(y)$ is $\Lambda$-semiconcave for all $n\in\mathbb{N}$ with $\Lambda= \norm{B_\infty\Sigma}_2\leq \sqrt{\nicefrac{\beta_\nu}{\alpha_\rho}}\|\Sigma\|_2$. Therefore in this case \Cref{cor:exp_conv_from_stab}-\ref{item:exp_conv_from_stab_i} gives the exponential convergence of Sinkhorn's algorithm with rate given by \eqref{eq:rate:caff} if $\varepsilon\leq \tau \,\sqrt{\nicefrac{\beta_\nu}{\alpha_\rho}}\,\norm{\Sigma}_2$. In fact, we have actually proven that the exponential convergence with rate~\eqref{eq:rate:caff} holds under the weaker assumption $\nabla^2\phizero{\varepsilon} \succeq - \Sigma -\varepsilon\frac{\alpha_\rho}{2}+ \Big(\frac{\varepsilon^2\alpha^2_\rho}{4}+ \frac{\alpha_\rho}{\beta_\nu}\Sigma^2 \Big)^{1/2}$ but we have preferred to keep the stronger condition in the statement of \Cref{prop:anisotropic:cost}-\ref{prop:anisotropic:cost:caff} for better readability.

\subsubsection{Marginal $\rho$ weakly log-concave}\label{app:weak:logcconcave}
Here we restrict to the quadratic cost setting  (\ie, $\Sigma=\Id$) and relax the convexity assumption of $U_\rho$ by simply requiring $\rho$ to be weakly log-concave (cf.\ \eqref{eq:weak_log_conc}), that is, we assume 
\be\label{eq:weakly_log_conc_in_proofs}
\langle \nabla U_\rho(\hat{x}) - \nabla U_{\rho}(x),\hat x - x\rangle \geq \alpha_\rho |\hat x -x |^2 - |\hat x -x |f_L(|\hat{x}-x|)\,, \quad \forall x,\hat x \in \bbR^d,
\ee
where for any $L\geq0$, $r>0$, $f_L(r)=2L^{1/2}\tanh\big((L^{1/2} r)/2\big)$. We now reformulate this condition introducing the convexity profile (this is a classical notion in the coupling literature, see \cite{eberle2016reflection} for example) of a given $U$ as the function $\kappa_{U}:\bbR_+\longrightarrow\bbR$ given by
\be\label{def:kapU::eberle}
\kappa_{U}(r)\coloneqq \inf\biggl\{\frac{\langle\nabla U(\hat{x})-\nabla U(x),\,\hat{x}-x\rangle}{|x-\hat{x}|^2} \quad\colon\, |x-\hat{x}|=r\biggr\}\,.
\ee
For ease of notation we also set $\kappa_\rho\coloneqq \kappa_{U_\rho}$. With this notation at hand, \eqref{eq:weakly_log_conc_in_proofs} rewrites as  \be\label{eq:kappa_in_proof} 
\kappa_{\rho}(r)\geq \alpha_\rho - r^{-1}f_L(r) \quad \forall r>0\,.
\ee
It follows from \eqref{eq:hess:hess:cov:phieot} that, since $\nabla^2_1c = \Id$,
\be\label{appo:lower:kappa:sink}
\kappa_{(\varepsilon^{-1}\phieot{n}+U_\rho)}(r)\geq \alpha_\rho-\varepsilon^{-1}-r^{-1} f_L(r)\, \quad \forall r>0, n\in\N\,.
\ee
At this point, we can invoke \cite[Lemma 3.1]{conforti2024weak} (see \Cref{appendix} for a statement of this result using the current notation) to conclude that 
\bes
\kappa_{\Phi^\rho_0(\phieot{n})}(r) \geq  \frac{\varepsilon \alpha_\rho -1}{\varepsilon\alpha_\rho}- \frac{L}{\varepsilon\alpha_\rho^2}\,, 
\ees
which gives 
\bes
\nabla^2\psieot{n}=-\nabla^2\Phi^\rho_0(\phieot{n})\preceq -\frac{\varepsilon\alpha_\rho-1}{\varepsilon\alpha_\rho}+ \frac{L}{\varepsilon\alpha_\rho^2}=-1+\frac1{\varepsilon\alpha_\rho}+\frac{L}{\varepsilon\alpha_\rho^2} \,.
\ees

From this we immediately conclude that uniformly in $x\in\supp(\rho)$, for any $n\in\N$, the map $y\mapsto c(x,y)+\psieot{n}(y)$ is $\Lambda$-semiconcave 
with $\Lambda=(\varepsilon\alpha_\rho)^{-1}+\nicefrac{L}{\varepsilon\alpha_\rho^2}$. We can thus invoke \Cref{cor:exp_conv_from_stab}-\ref{item:exp_conv_from_stab_i} which gives exponential convergence of Sinkhorn's algorithm. Moreover, if $\varepsilon\leq \alpha_\rho^{-1}\sqrt{\tau(\alpha_\rho+L)}$, then the convergence rate is
\bes
1-\frac{\varepsilon^2\alpha_\rho^2}{\varepsilon^2\alpha_\rho^2+\tau(\alpha_\rho+L)}\,.
\ees

\begin{flushright}
    $\square$
\end{flushright}

\subsection{Proof of Proposition \ref{prop:light:tails}} 
Assume that the cost satisfies 
\bes
\nabla_1^2c(x,y)\,,\nabla^2_2c(x,y)\succeq\ellc\quad\text{ and }\quad\norm{\nabla_{12}^2c(x,y)}_2\vee\norm{\nabla_{1}^2c(x,y)}_2\leq\Lc\,,
\ees
uniformly in $x\in\bbRD$ and $y\in\cY$. Moreover assume that $\rho$ has light tails in the following sense
\bes
\nabla^2U_\rho(x)\succeq\begin{cases}
    C\,|x|^{\delta}\quad&\text{ for }|x|>R\,,\\
    -L\quad&\text{ for }|x|\leq R\,,
\end{cases} 
\ees
for some positive $\delta,\,C>0$ and $L,\,R\geq 0$. To estimate the conditional covariances, we shall use the following abstract result.
 
\begin{lemma}\label{lem:conv_lip_pert}
Let $\pi=e^{-U}$, $\alpha>0,L,R\geq0$ be such that
\be\label{ass:conv_lip_perp}
\nabla^2 U(x)\succeq 
\begin{cases} 
-L, \quad & {|x|< R}\,, \\ 
\alpha, \quad & |x|\geq R\,.
\end{cases}
\ee
Moreover, let $f:\bbRD\rightarrow\bbRD$ be Lipschitz . Then we have
\bes
\mathrm{Cov}_{X\sim\pi}(f(X)) \preceq 2\mathrm{Lip}(f)^2 \big(\alpha^{-1}+\alpha^{-2}(L+\alpha)^2R^2\big)\,.
\ees
\end{lemma}

\begin{proof}
We define the function $\phi:\mathbb{R}^d\rightarrow \mathbb{R}$
given by
\bes
\phi(x)= \frac{L+\alpha}{2} |x|^2 \mathbf{1}_{\{|x|<R\}}+  \Big((L+\alpha)R(|x|-R) +\frac{(L+\alpha)R^2}{2} \Big)\mathbf{1}_{\{|x|\geq R\}}\,.
\ees
Then, the function $\phi$ is convex and $(L+\alpha)R$-Lipschitz. Moreover, $\nabla^2\phi(x)\succeq (L+\alpha)$ for $|x|<R$. But then, we deduce from \eqref{ass:conv_lip_perp} that $\bar{\pi}(\De x)\propto\exp(-U(x)-\phi(x))\De x$ is $\alpha$-log-concave and in particular satisfies $\mathrm{PI}(\alpha^{-1})$. Now, let $(X,\bar X)$ be a coupling between $\pi$ and $\bar\pi$. Then, following the argument given in \cite[Lemma 2.1]{brigati2024heat}, we find that for all $v$ such that $|v|=1$ 
\be\label{eq:con_lip_per_cov_est}
\begin{aligned}
\langle v,\,\mathrm{Cov}_{X\sim \pi}(f(X)) v\rangle &= \mathrm{Var}_{X\sim \pi}(\langle f(X),v\rangle)\\
&\leq 2\mathrm{Var}_{\bar{X}\sim \bar\pi}(\langle f(\bar X),v\rangle) + 2\bbE[|\langle f(X),v\rangle -\langle f(\bar{X}),v\rangle|^2]\,.
\end{aligned}
\ee 
Using that the function $x\mapsto \langle f(x),v\rangle$ is $\mathrm{Lip}(f)$-Lipschitz, we can use that $\bar\pi$ satisfies $\mathrm{PI}(\alpha^{-1})$ to obtain 
\bes 
\mathrm{Var}_{\bar{X}\sim \bar\pi}(\langle f(\bar X),v\rangle)\leq \mathrm{Lip}(f)^2\alpha^{-1}\,.
\ees 
Moreover, if $(X,\bar X)$ is optimal for $\bfW_2(\pi,\bar\pi)$ we obtain 
\bes
\bbE[|\langle f(X),v\rangle -\langle f(\bar{X}),v\rangle|^2] \leq \mathrm{Lip}(f)^2\bbE[\sfd(X,\bar{X})^2] = \mathrm{Lip}(f)^2 \bfW_2^2(\pi,\hat\pi)\,.
\ees
Plugging the last two bounds back into \eqref{eq:con_lip_per_cov_est} and invoking first TI($\alpha^{-1}$) then LSI($\alpha^{-1}$) for $\bar\pi$ we obtain
\bes
\langle v,\,\mathrm{Cov}_{f(X)\sim \pi}(f(X)), v\rangle \leq 2 \mathrm{Lip}(f)^2 \big(\alpha^{-1} + \alpha^{-2}  \mathrm{Lip}(\phi)^2 \big)\,.
\ees
Since $\mathrm{Lip}(\phi) \leq (L+\alpha)R$, the conclusion follows.
\end{proof}

With this result at hand, we obtain the following.

\begin{lemma}\label{lemma:light:bound:alpha}
Assume that $c$ satisfies \eqref{eq:ass:hessia:bounded} and that $\rho$ satisfies \eqref{eq:ass:light:tails}. Then for any $\alpha>0$ we have 
\be\label{cov:alpha:bound}
\sup_{y\in\cY}\norm{\mathrm{Cov}_{X\sim\pi^{n,n}(\cdot|y)}(\nabla_2c(X,y))}_2\leq 2\Lc^2\,\biggl(\frac\varepsilon\alpha+\biggl(\frac{L\varepsilon}{\alpha}+\frac{\alpha+\delL}{\alpha} \biggr)^2\,\biggl[R^2\vee\biggl(\frac{\alpha+\delL}{\varepsilon\,C}\biggr)^\frac2\delta\biggr]\biggr)\,,
\ee
where we have set $\delL\coloneqq \Lc-\ellc$. As a consequence of this we have
\bes
\sup_{y\in\cY}\norm{\mathrm{Cov}_{X\sim\pi^{n,n}(\cdot|y)}(\nabla_2c(X,y))}_2\leq 2\Lc^2\,\biggl(1+(L+2)^2\biggl[R^2\vee {C^{-\nicefrac2\delta}}\biggl(1+2\frac{\delL}{\varepsilon}\biggr)^{\nicefrac2\delta}\biggr]\biggr)\,.
\ees
\end{lemma}

\begin{proof}
Notice that the upper bound on the Hessian of the cost guarantees the semiconvexity of $\phieot{n}$, since we recall from \eqref{eq:hess:hess:cov:phieot} that
\bes
\nabla^2\phieot{n}(x)=-\int_\cY\nabla_1^2 c(x,y)\,\pi^{n,n-1}(\De y|x)+\varepsilon^{-1} \,\mathrm{Cov}_{Y\sim\pi^{n,n-1}(\cdot|x)}(\nabla_1 c(x,Y))\succeq -\Lc\,.
\ees
By using this estimate we deduce that uniformly in $y\in\cY$ it holds
\bes
\nabla^2_1(-\log\pi^{n,n}(x|y))\succeq \begin{cases}
    C\,|x|^\delta-\frac{\delta_H}{\varepsilon}\quad&\text{ for }|x|>R\,,\\
    -L-\frac{\delta_H}{\varepsilon}\quad&\text{ for }|x|\leq R\,. 
\end{cases}
\ees
Therefore, for any fixed $\alpha>0$, if we set
\bes
R(\alpha,\varepsilon)=R\vee\biggl(\frac{\alpha+\delta_H}{\varepsilon\,C}\biggr)^\frac1\delta \,,
\ees
we find 
\bes
\nabla^2_1(-\log\pi^{n,n}(x|y))\succeq \begin{cases}
    \frac{\alpha}{\varepsilon}\quad&\text{ for }|x|>R(\alpha,\varepsilon)\,,\\
    -L-\frac{\delta_H}{\varepsilon}\quad&\text{ for }|x|\leq R(\alpha,\varepsilon)\,. 
\end{cases}
\ees
We can now apply \Cref{lem:conv_lip_pert} to $\pi^{n,n}(\cdot|y)$ with $f(x)=\nabla_2c(x,y)$ (for which $\mathrm{Lip}(f) \leq \Lc$) to obtain that, uniformly in $y\in\cY$, it holds
\bes
\mathrm{Cov}_{X\sim\pi^{n,n}(\cdot|y)}(\nabla_2 c(X,y))\preceq 2\Lc^2\,\biggl(\frac\varepsilon\alpha + \biggl( \frac{L\varepsilon+\delta_H}{\alpha}+1\biggr)^2\,\biggl[R^2\vee\biggl(\frac{\alpha+\delta_H}{\varepsilon\,C} \biggr)^\frac2\delta\biggr]\biggr)\,.
\ees
We have thus shown the first bound. The second one can be obtained by considering the specific choice $\bar\alpha=\varepsilon+\delL$.
\end{proof}
As a corollary of the previous estimate we finally deduce the convergence rate for Sinkhorn's algorithm in the light-tails regime.

\begin{corollary}
Assume that $c$ satisfies \eqref{eq:ass:hessia:bounded} and that $\rho$ satisfies \eqref{eq:ass:light:tails}. Then Sinkhorn's algorithm converges exponentially fast. Moreover, if $\Lc=\ellc$ then, for 
\be\label{esp:small:light:tails:del=0}
\varepsilon\leq \sqrt{2\tau\Lc^2\,(1+(L+2)\,[R^2\vee C^{-\nicefrac2\delta}])}
\ee
the rate is
\bes
1-\frac{\varepsilon^2}{\varepsilon^2+2\tau\Lc^2(1+(L+2)^2\,[R^2\vee C^{-\nicefrac2\delta}])}\,.
\ees
Otherwise, if $\Lc>\ellc$ (and set $\nicefrac{1}{0}=+\infty$) then for
\be\label{esp:small:light:tails:del:not:0}
\varepsilon\leq 1\wedge\frac{\delL}{(R^{\delta} C-1)_+}\wedge\biggl(\tau\delL+2\tau\Lc^2\,(1+(L+2)^2\,C^{-\nicefrac2\delta}(1+2\delL)^{\nicefrac2\delta}\biggr)^{\nicefrac{\delta}{2+2\delta}}
\ee
we have rate of convergence 
\bes
1-\frac{\varepsilon^{2+\nicefrac2\delta}}{\varepsilon^2+\tau\varepsilon\delL+2\tau\Lc^2\,(1+(L+2)^2\,C^{-\nicefrac2\delta}(\varepsilon+2\delL)^{\nicefrac2\delta})}\,.
\ees
\end{corollary}

\begin{proof}
From \Cref{lemma:light:bound:alpha} and \eqref{eq:hess:hess:cov:psieot} we deduce that
\bes
\begin{aligned}
\nabla^2_2(c(x,y)+\psieot{\nu}(y))\preceq&\, \Lc-\ellc+\varepsilon^{-1}\sup_{y\in\cY}\norm{\mathrm{Cov}_{X\sim\pi^{n,n}(\cdot|y)}(\nabla_2c(X,y))}_2\\
\leq&\,\delL+ \frac{2\Lc^2}{\varepsilon}\,\biggl(1+(L+2)^2\biggl[R^2\vee {C^{-\nicefrac2\delta}}\biggl(1+2\frac{\delL}{\varepsilon}\biggr)^{\nicefrac2\delta}\biggr]\biggr)\eqqcolon \Lambda\,.
\end{aligned}
\ees
This proves that $y\mapsto c(x,y)+\psieot{n}(y)$ is $\Lambda$-semiconcave uniformly in $n\in\N$ and in $x\in\bbRD$. Thus, from \Cref{cor:exp_conv_from_stab} we deduce the convergence of Sinkhorn's algorithm and its rate. On the one hand, when $\delL=\Lc-\ellc=0$, for $\varepsilon$ small enough as in \eqref{esp:small:light:tails:del=0}, the rate is
\bes
1- \frac{\varepsilon^2}{\varepsilon^2+2\tau\Lc^2\,(1+(L+2)^2\,[R^2\vee C^{-\nicefrac2\delta}])}\,.
\ees
On the other hand, if $\Lc>\ellc$ then, eventually for $\varepsilon$ small enough, we have 
\bes
R^2\leq C^{-\nicefrac2\delta}\biggl(1+2\frac{\delL}{\varepsilon}\biggr)^{\nicefrac2\delta} \,.
\ees
Hence \Cref{cor:exp_conv_from_stab} implies the convergence of Sinkhorn's algorithm and for $\varepsilon$ satisfying \eqref{esp:small:light:tails:del:not:0} the rate is 
\bes
1-\frac{\varepsilon^{2+\nicefrac2\delta}}{\varepsilon^{2+\nicefrac2\delta}+\varepsilon^{1+\nicefrac2\delta}\tau\delL+2\tau\Lc^2\,(\varepsilon^{\nicefrac2\delta}+(L+2)^2\,C^{-\nicefrac2\delta}(\varepsilon+2\delL)^{\nicefrac2\delta})}\,
\ees
which for reader's clarity we have simplified in the final statement by imposing $\varepsilon\leq 1$.
\end{proof}

\subsubsection{Quadratic cost} For the Euclidean quadratic cost, the previous discussion (with $\Lc=\ellc=1$) combined with \Cref{cor:exp_conv_from_stab} shows the exponential convergence of Sinkhorn's algorithm when the tails of $\rho$ are light. More precisely, if \eqref{eq:ass:light:tails} holds for some $C,\,\delta>0$ and $R,\,L\geq 0$, then Sinkhorn's algorithm exponential converges with rate $1-\Theta(\nicefrac{\varepsilon^2}{\tau(1+L^2\,[R^2\vee C^{-\nicefrac2\delta}])})$ as $\varepsilon\rightarrow0$, see \Cref{table:comparison} for the meaning of $\Theta(\cdot)$.

\subsubsection{Anisotropic quadratic costs and subspace elastic costs}
As we have already noticed in \eqref{anisotropic:covariance}, when considering the anisotropic quadratic cost $c(x,y)\coloneqq \langle x-y,\,\Sigma (x-y)\rangle/2$ we have 
\bes
\nabla_2^2(c(x,y)+\psieot{n}(y))=\Sigma+\nabla^2\psieot{n}(y)\overset{\eqref{eq:hess:hess:cov:psieot}}{=}\varepsilon^{-1} \,\mathrm{Cov}_{X\sim\pi^{n,n}(\cdot|y)}(\Sigma(y-X))\,,
\ees
and hence it is enough to bound uniformly in $y\in\bbRD$, for any unit vector $v$, the variance $\mathrm{Var}_{X\sim\pi^{n,n}(\cdot|y)}(\langle v,\,\Sigma X \rangle)$. This can be done via Poincar\'e inequality, as explained above with $\rho$ satisfying the light-tails condition \eqref{eq:ass:light:tails}. More precisely, 
we can reason as in \Cref{lemma:light:bound:alpha} to deduce that
\bes
\varepsilon\,\langle v,\,\nabla_2^2(c(x,y)+\psieot{n}(y))v\rangle = \,\mathrm{Var}_{\pi^{n,n}(\cdot|y)}(\langle v,\Sigma\,X\rangle)\leq 2\|\Sigma\|_2^2\,\bigl(1+(L+1)^2\bigl[R^2\vee {C^{-\nicefrac2\delta}}\bigr]\bigr)\,.
\ees
From this, we conclude that $y\mapsto c(x,y)+\psieot{n}(y)$ is $\Lambda$-semiconcave uniformly in $n\in\N$ and in $x\in\bbRD$, with  $\Lambda=2\varepsilon^{-1}\,\|\Sigma\|_2^2\,(1+(L+1)^2[R^2\vee {C^{-\nicefrac2\delta}}])$, which combined with our main result \Cref{cor:exp_conv_from_stab} implies the exponential convergence of Sinkhorn's algorithm and for 
\be\label{esp:small:light:tails:del=0:Sigma}
\varepsilon\leq \norm{\Sigma}_2\,\sqrt{2\,\tau(1+(L+1)^2[R^2\vee {C^{-\nicefrac2\delta}}]) }
\ee
the rate is 
\bes
1-\frac{\varepsilon^2}{\varepsilon^2+2\,\tau\,\|\Sigma\|_2^2\,\bigl(1+(L+1)^2\bigl[R^2\vee {C^{-\nicefrac2\delta}}\bigr]\bigr)}\,.
\ees

\subsection{Proof of Proposition \ref{prop:lip:costs}}
Assume that 
\bes
\ellc\preceq\nabla^2_2c(x,y)\preceq\Lc\quad\text{ and }\quad\nabla^2_1c(x,y)\preceq \Lc\,.
\ees
We distinguish two cases.

\subsubsection{Cost Lipschitz w.r.t.\ $y\in\cY$}

If the cost is Lipschitz in $y$, uniformly in $x$, that is,
\bes
\sup_{x\in\bbRD}|c(x,y)-c(x,\hat{y})|\leq \mathrm{Lip}_{\infty,2}(c)\,|y-\hat{y}| \quad \forall y,\hat y \in \mathbb{R}^d\,,
\ees
then we deduce from \eqref{eq:hess:hess:cov:psieot} that
\bes
\nabla^2_2(c(x,y)+\psieot{n}(y))\preceq \Lc-\ellc+\varepsilon^{-1} \,\mathrm{Lip}_{\infty,2}^2(c)\,.
\ees
Therefore $y\mapsto c(x,y)+\psieot{n}(y) $ is $\Lambda$-semiconcave uniformly in $x\in\bbRD$ and in $n\in\N$ with $\Lambda= \Lc-\ellc+\varepsilon^{-1} \,\mathrm{Lip}^2_{\infty,2}(c)$. \Cref{cor:exp_conv_from_stab} implies then the convergence of Sinkhorn's algorithm and for 
\be\label{eps:small:lip2:cost}
\varepsilon\leq\frac{\tau(\Lc-\ellc)+\sqrt{\tau^2(\Lc-\ellc)^2+4\tau \mathrm{Lip}_{\infty,2}(c)^2}}{2}
\ee
the rate is
\bes
1-\frac{\varepsilon^2}{\varepsilon^2+\varepsilon\tau(\Lc-\ellc)+\tau\mathrm{Lip}_{\infty,2}(c)^2}\,.
\ees

\subsubsection{Cost Lipschitz w.r.t.\ $x\in\cX$}

Assume our cost is Lipschitz in $x$, uniformly in $y$, \ie, such that
\bes
\sup_{y\in\bbRD}|c(x,y)-c(\hat{x},y)|\leq \mathrm{Lip}_{1,\infty}(c)\,|x-\hat{x}| \quad \forall x,\hat{x}\in\mathbb{R}^d.
\ees
Moreover, we further assume that $\norm{\nabla_{12}^2c(x,y)}_2\leq \Lc$ and that $\rho$ satisfies \ref{eq:LSI}.

Notice that the Lipschitz continuity of the cost directly propagates to the Sinkhorn iterate $\phieot{n}$ since $\nabla\phieot{n}(x)=-\int\nabla_1 c(x,y)\,\pi^{n,n-1}(\De y|x)$ and hence $\mathrm{Lip}(\phieot{n})\leq \mathrm{Lip}_{1,\infty}$, see also \cite[Prop 2.4]{marinogerolin2020}.
This key observation allows us to prove the following lemma.

\begin{lemma}
Assume $c$ satisfies \eqref{eq:def:c:lip:unif:x} and that $\norm{\nabla_{12}^2c(x,y)}_2 \vee \norm{\nabla_{1}^2c(x,y)}_2\leq\Lc$ and further that $\rho$ satisfies \ref{eq:LSI}. Then 
\bes
\sup_{y\in\cY}\norm{\mathrm{Cov}_{X\sim\pi^{n,n}(\cdot|y)}(\nabla_2c(X,y))}_2\leq 2\Lc^2C_\rho\,(1+\nicefrac{4C_\rho\mathrm{Lip}_{1,\infty}^2}{\varepsilon^2})\,.
\ees
\end{lemma}

\begin{proof}
The proof resembles to that of \Cref{lemma:light:bound:alpha}. Arguing as we did there, for any $v$ with $|v|=1$ the Poincar\'e inequality \ref{def:poincare:ineq} and the Talagrand $\mathrm{TI}(C_\rho)$ combined with the $\mathrm{LSI}(C_\rho)$ yield
\bes
\begin{aligned}
\langle v,\,\mathrm{Cov}_{X\sim\pi^{n,n}(\cdot|y)}(\nabla_2c(X,y))\,v\rangle\leq&\, 2\Lc^2\,\bfW_2^2(\pi^{n,n}(\cdot|y),\rho)+2\,\mathrm{Var}_{\hat{X}\sim\rho}(\langle v,\,\nabla_2c(\hat{X},y)\rangle)\\
\leq&\, 4\Lc^2 \,C_\rho\,\scrH(\pi^{n,n}(\cdot|y)|\rho)+2 C_\rho\int|\nabla_{12}^2 c(x,y)v|^2\De\rho\\
\leq&\, 2\Lc^2C_\rho\biggl(\frac{C_\rho}{\varepsilon^2}\int|\nabla_1(c(x,y)+\phieot{n}(x))|^2\pi^{n,n}(\De x|y)+1\biggr)\\
\leq&\, 2\Lc^2C_\rho\,(1+\nicefrac{4C_\rho\mathrm{Lip}^2_{1,\infty}}{\varepsilon^2})\,.
\end{aligned}
\ees
\end{proof}

As a corollary of the previous estimate, we finally deduce the convergence rate for Sinkhorn's algorithm in this setting. In fact, it is enough to notice that from \eqref{eq:hess:hess:cov:psieot} we have 
\bes
\begin{aligned}
\nabla^2_2(c(x,y)+\psieot{n}(y))\preceq&\, \Lc-\ellc+\varepsilon^{-1}\,\mathrm{Cov}_{X\sim\pi^{n,n}(\cdot|y)}(\nabla_2 c(X,y))\\
\preceq&\,\Lc-\ellc+\frac{2\Lc^2\,C_\rho}{\varepsilon}\,(1+\nicefrac{4C_\rho\mathrm{Lip}_{1,\infty}^2}{\varepsilon^2})\,.
\end{aligned}
\ees
Therefore, the map $y\mapsto c(x,y)+\psieot{n}(y) $ is $\Lambda$-semiconcave uniformly in $x\in\bbRD$ and in $n\in\N$ with $\Lambda= \Lc-\ellc+\frac{2\Lc^2\,C_\rho}{\varepsilon}\,(1+\nicefrac{4C_\rho\mathrm{Lip}_{1,\infty}^2}{\varepsilon^2})$. \Cref{cor:exp_conv_from_stab} implies then the convergence of Sinkhorn's algorithm and for 
\be\label{eps:small:lip1}
\varepsilon\leq 1\wedge \biggl(\tau(\Lc-\ellc)+2\tau\Lc^2C_\rho(1+4C_\rho\mathrm{Lip}_{1,\infty}^2)\biggr)^{\nicefrac14}
\ee
the rate is  
\bes
 1-\frac{\varepsilon^4}{\varepsilon^4+\tau\varepsilon^3(\Lc-\ellc)+2\tau\Lc^2C_\rho(\varepsilon^2+4C_\rho\mathrm{Lip}_{1,\infty}^2)}\,.
\ees
\begin{flushright}
    $\square$
\end{flushright}

\subsection{Proof of Proposition \ref{prop:compact:supp:sink}}

Whenever we assume that $\rho$ is compactly supported, we aim to estimate the $\Lambda$-semiconcavity of $y\mapsto c(x,y)+\psieot{n}(y)$  by relying on \eqref{eq:hess:hess:cov:psieot} and the fact that the covariance under the conditional probability $\pi^{n,n}(\cdot|y)$ has bounded  support, \ie\ $\supp(\pi^{n,n}(\cdot|y))\subseteq \supp(\rho)$. In order to do that we need an additional assumption either on the cost function or on the second marginal $\nu$.

\medskip

\subsubsection{Case $\nabla_2^2c(x,y)=\Sigma$}\label{app:phi:compact}  Under this additional assumption, \eqref{eq:hess:hess:cov:psieot} reads as
\bes
\nabla^2(c(x,y)+\psieot{n}(y))=\varepsilon^{-1} \,\mathrm{Cov}_{X\sim\pi^{n,n}(\cdot|y)}(\nabla_2 c(X,y))\,,
\ees
 which means that we can compute the semiconcavity parameter $\Lambda$ by estimating
\bes
\sup_{y\in\cY}\norm{\mathrm{Cov}_{X\sim\pi^{n,n}(\cdot|y)}(\nabla_2 c(X,y))}_2\,.
\ees
To bound this, note that $\nabla_2^2c(x,y)=\Sigma$ implies that $\nabla_2 c(x,y)=\Sigma\,y+g(x)$ is affine w.r.t.\ $y$ and hence
\bes
\sup_{y\in\cY}\norm{\mathrm{Cov}_{X\sim\pi^{n,n}(\cdot|y)}(g(X))}_2\leq \max_{x\in\supp(\rho)} \norm{g(x)}_2^2=\norm{g}_{\rmL^\infty(\rho)}^2\,.
\ees
This implies that, uniformly in $x\in\supp(\rho)$ for every $n\in\N$, the function $y\mapsto c(x,y)+\psieot{n}(y)$ is $\Lambda$-semiconcave with parameter $\Lambda \leq \varepsilon^{-1} \norm{g}_{\rmL^\infty(\rho)}^2$, which combined with \Cref{cor:exp_conv_from_stab} implies the exponential convergence of Sinkhorn's algorithm and for $\varepsilon\leq \sqrt{\tau}\,\norm{g}_{\rmL^\infty(\rho)}$ the rate is
\bes
1-\frac{\varepsilon^2}{\varepsilon^2+\tau \|g\|_{\rmL^\infty(\rho)}^2}\,.
\ees
In particular, note that the convergence rate proven here is independent of the matrix $\Sigma$.

This setting covers the quadratic cost $c(x,y)=|x-y|^2/2$ over $\bbRD$, where $\Sigma=\Id$, $\nabla_2c(x,y)=y-x$ and hence for any $\rho$ such that $\supp \rho \subseteq B_R(0)$ it holds $\Lambda\leq \nicefrac{R^2}{\varepsilon}$. Thus,
the convergence rate of Sinkhorn's algorithm, for $\varepsilon$ small enough, is $1-\Theta(\nicefrac{\varepsilon^2}{\tau R^2})$, see \Cref{table:comparison} for the meaning of $\Theta$. Other relevant examples covered here are the anisotropic quadratic costs and subspace elastic costs (already considered in \Cref{prop:anisotropic:cost} in the weakly log-concave regime).

\subsubsection{If $\nu$ compactly supported}\label{app:compact:compact}
From the compactness of both marginals we are guaranteed that uniformly in $\supp(\rho)\times\supp(\nu)$ it holds
\bes
\ellc\preceq \nabla_2^2c(x,y)\preceq \Lc\,,
\ees
for some finite $\ellc,\,\Lc$. Hence from \eqref{eq:hess:hess:cov:psieot}
\bes
\nabla^2\psieot{n}(y)\preceq -\ellc+\varepsilon^{-1}\max_{y\in\supp(\nu)}\norm{\mathrm{Cov}_{X\sim\pi^{n,n}(\cdot|y)}(\nabla_2 c(X,y))}_2\preceq -\ellc+\varepsilon^{-1}\norm{\nabla_2 c}_{\rmL^\infty(\rho\times\nu)}^2\,,
\ees
which yields
\bes
\nabla^2(c(x,y)+\psieot{n}(y))
\preceq \Lc-\ellc+\varepsilon^{-1}\norm{\nabla_2 c}_{\rmL^\infty(\rho\times\nu)}^2\,.
\ees
This, combined with \Cref{cor:exp_conv_from_stab}, implies the exponential convergence of Sinkhorn's algorithm and for 
\be\label{eps:small:compact:both}
\varepsilon\leq \frac{\tau(\Lc-\ellc)+\sqrt{\tau^2(\Lc-\ellc)^2+4\tau \norm{\nabla_2 c}_{\rmL^\infty(\rho\times\nu)}^2}}{2}
\ee
 the rate is
\bes
1-\frac{\varepsilon^2}{\varepsilon^2+\tau\varepsilon(\Lc-\ellc)+\tau\norm{\nabla_2 c}_{\rmL^\infty(\rho\times\nu)}^2}\,.
\ees

\begin{flushright}
    $\square$
\end{flushright}

\subsection{Proof of Proposition \ref{prop:manifolds}}
Let us briefly discuss how the gradient and the Hessian of any function $f\colon\sSd\to \rset$ can be computed in the sphere Riemannian metric. For notations' clarity, in the Riemannian setting we denote gradient and Hessian with $\nabla_\sSd$ and $\Hess_\sSd$, respectively. In view of this, let us fix $y\in\sSd$ as well as a tangent vector $v \in T_y \mathbb{S}^d \hookrightarrow \mathbb{R}^{d+1}$. This allows us to consider the constant-speed geodesic started at $y \in \mathbb{S}^d$ with velocity $v \in T_y \mathbb{S}^d$ as the curve $(\gamma_t)_{t \in [0,1]}$ given by $\gamma_t = \exp_y (tv)$, explicitly given by 
\be\label{eq:geodesic_sphere}
\gamma_t = \exp_y(tv) = \cos(t\|v\|)\,y + \sin(t\|v\|)\frac{v}{\|v\|}\,.
\ee
To determine $\nabla_\sSd f(y)\in T_y\sSd$ and $\Hess_{\sSd} f(y) : (T_y\sSd)^{\otimes 2} \to \mathbb{R}$ it is enough to compute 
\be\label{eq:first_second_deriv}
(f\circ \gamma_t)^\prime(0)=\langle \nabla_\sSd f(y),
v\rangle_{\mathfrak{g}} \quad\text{ and }\quad
(f\circ\gamma_t)^{\prime\prime}(0)=\Hess_{\sSd} f(y)(v,v)\,.
\ee
When considering the function $f(y)=c(x,y) = 1-\langle x,y \rangle$, with $x\in\sSd$ fixed, it is not difficult to see that
\bes
\nabla_\sSd c(x,\cdot)(y) = -(\Id-yy^\intercal)x \qquad \textrm{and} \qquad \Hess_\sSd c(x,\cdot)(y) = \langle x,y\rangle \Id\,.
\ees
This immediately entails the uniform bound
\bes
-\mathfrak{g}\, \preceq \Hess_\sSd c(x,\cdot) \preceq \mathfrak{g}\,,
\ees
with $\mathfrak{g}$ being the canonical Riemannian metric of $\sSd$ given by the inclusion $\sSd \hookrightarrow \mathbb{R}^{d+1}$, from which we may deduce that on $T_y\sSd$ it holds
\bes 
\begin{aligned}
\Hess_\sSd \psieot{n}(y) \preceq \Id + \varepsilon^{-1}\,\mathrm{Cov}_{X\sim\pi^{n,n}(\cdot|y)}\biggl((\Id-yy^\intercal)X \biggr) \preceq 1 + \varepsilon^{-1}\,.
\end{aligned}
\ees
Therefore the map $y\mapsto c(x,y)+\psieot{n}(y)$ is $\Lambda$-semiconcave uniformly in $x\in\sSd$ and $n\in\N$ with $\Lambda=2+\varepsilon^{-1}$,
which allows us to apply \Cref{cor:exp_conv_from_stab} and deduce the exponential convergence of Sinkhorn's algorithm. 
Moreover, for any $\varepsilon\leq\tau+\sqrt{\tau+\tau^2} $ the exponential convergence rate is equal to $1-\nicefrac{\varepsilon^2}{\varepsilon^2+2\tau\varepsilon+\tau}$. This proves the first part of \Cref{prop:manifolds}.

\medskip

Along the same lines, if we now consider $f(y)\coloneqq c_\delta(x,y) = \arccos(\delta\langle x,y \rangle)^2$, with $x\in\sSd$ fixed, and compute again \eqref{eq:first_second_deriv}, straightforward computations lead to 
\bes
\nabla_\sSd c_\delta(x,\cdot)(y)=-\frac{2\delta\,\arccos(\delta\langle x,y\rangle)}{\sqrt{1-\delta^2\langle x,y\rangle^2}}(\Id-yy^\intercal)x
\ees
and 
\be\label{eq:expl:hess:sphere}
\begin{split}
\Hess_\sSd c_\delta(x,\cdot)(y) & = \frac{2\delta\,\arccos(\delta\langle x,y\rangle)}{\sqrt{1-\delta^2\langle x,y\rangle^2}}\langle x,y\rangle \Id \\
& \quad + 2\delta^2\frac{1-\langle x,y\rangle^2}{1-\delta^2\langle x,y\rangle^2}\,\biggl(1-\delta\frac{\langle x,y\rangle\,\arccos(\delta\langle x,y\rangle)}{\sqrt{1-\delta^2\langle x,y\rangle^2}}\biggr)\,u\otimes u\,,
\end{split}
\ee
where $u$ is a unit-norm vector defined as $u\coloneqq \nicefrac{(\Id-yy^\intercal)x}{\norm{(\Id-yy^\intercal)x}}\in T_y\sSd$. Let us also remark here that in the limit $\delta\uparrow 1$ the above quantities give the gradient and the Hessian of the squared distance function $y\mapsto \sfd^2(x,y)$ (in total concordance with the convergence $c_\delta\to\sfd^2$ as $\delta\uparrow1$).

Given the above preliminary remarks, we are now ready to prove \Cref{prop:manifolds}. Indeed, from~\eqref{eq:expl:hess:sphere} uniformly in $x\in\sSd$ we have
\bes
-\frac{2\pi}{\sqrt{1-\delta^2}}\,\mathfrak{g}\,\preceq \Hess_\sSd c_\delta(x,\cdot)\preceq \biggl(2\delta^2 +\frac{2\pi}{\sqrt{1-\delta^2}}\biggr)\,\mathfrak{g}\,.
\ees
From these bounds it is also clear why we are restricting to costs $c_\delta$ with $\delta<1$.

As a consequence of the above discussion and \eqref{eq:hess:hess:cov:psieot} we have
\bes
\begin{aligned}
\Hess_\sSd \psieot{n}(y)\preceq\frac{2\pi}{\sqrt{1-\delta^2}}+\varepsilon^{-1}\,\mathrm{Cov}_{X\sim\pi^{n,n}(\cdot|y)}\biggl(\frac{2\delta\,\arccos(\delta\langle X,y\rangle)}{\sqrt{1-\delta^2\langle X,y\rangle^2}}(\Id-yy^\intercal)X\biggr)\\
\preceq \frac{2\pi}{\sqrt{1-\delta^2}}+\varepsilon^{-1}\frac{4\pi^2}{1-\delta^2}\,.
\end{aligned}
\ees
From this we immediately deduce that the map $y\mapsto c_\delta(x,y)+\psieot{n}(y)$ is $\Lambda$-semiconcave uniformly in $x\in\sSd$ and $n\in\N$ with 
\bes
\Lambda=2\delta^2 +\frac{4\pi}{\sqrt{1-\delta^2}}+\varepsilon^{-1}\frac{4\pi^2}{1-\delta^2}\,,
\ees
which allows us to apply \Cref{cor:exp_conv_from_stab} and deduce the exponential convergence of Sinkhorn's algorithm. Moreover, as soon as 
\be\label{eps:small:sfera:delta}
\varepsilon\leq \tau\biggl(\delta^2+\frac{2\pi}{\sqrt{1-\delta^2}}\biggr)+\sqrt{\tau^2\biggl(\delta^2+\frac{2\pi}{\sqrt{1-\delta^2}}\biggr)^2+\frac{4\pi^2\tau}{1-\delta^2}}
\ee
the convergence rate is 
\bes
1-\frac{\varepsilon^2}{\varepsilon^2+2\tau\varepsilon\left(\delta^2+\frac{2\pi}{\sqrt{1-\delta^2}}\right)+\frac{4\pi^2\tau}{1-\delta^2}}\,.
\ees
This concludes the proof of the second half of \Cref{prop:manifolds}.
\begin{flushright}
    $\square$
\end{flushright}

\subsection{Proof of Proposition \ref{prop:manifoldsSP}}

To establish this result, we first state an $\rmL^\infty$-bound for the gradient of the logarithm of the heat kernel. To be consistent with the assumptions of Proposition \ref{prop:manifoldsSP}, we assume the manifold to be compact, but the same conclusion holds more generally on non-compact manifolds with non-negative Ricci curvature.

\begin{lemma}\label{lem:hamilton_heatkernel}
Let $(M,\mathfrak{g})$ be a compact smooth Riemannian manifold. Then there is a constant $C$ depending on the Ricci curvature and on the dimension of $M$ only such that 
\be\label{eq:gradient_heat_kernel}
t^2|\nabla_M\log \msp_t(x,\cdot)(y)|^2 \leq 
C(1+\kappa^- t)(1+t), \qquad \forall x,y \in M,\, \forall t > 0\,,
\ee
where $\kappa \in \mathbb{R}$ is a lower bound for the Ricci curvature.
\end{lemma}

If ${\textrm{Ric}}_M \succeq 0$, then \eqref{eq:gradient_heat_kernel} has been proved in \cite{Kotschwar07}. If instead ${\textrm{Ric}}_M \succeq \kappa\mathfrak{g}$ with $\kappa < 0$, then a slight adaptation of \cite[Theorem 4.4]{gigli2017second} yields the same conclusion. For completeness, we postpone the proof to Appendix \ref{appendix2}. From this we immediately deduce that for any unit vector $v\in T_y M$
\bes
\varepsilon^{-1} \,\mathrm{Cov}_{X\sim\pi^{n,n}(\cdot|y)}(\langle\nabla_M c_\varepsilon(X,\cdot)(y),v\rangle_\mathfrak{g})\leq C(1+\kappa^-(1+ \varepsilon)+\varepsilon^{-1}) \,.
\ees

Next, we leverage a two-sided bound for the Hessian of the logarithm of the heat kernel obtained in \cite[Eq.(0.3)]{Stroock1996}, valid for $\varepsilon \leq 1$ (whence the necessity of this assumption). This result ensures the existence of a positive constant $C'$ depending only on the Ricci curvature, the sectional curvatures, and on the dimension of $M$ such that
\be\label{doppio:hess:c:eps}
- C'\left(1+\frac{{\textrm{diam}}(M)^2}{\varepsilon}\right) \mathfrak{g} \preceq \Hess_M c_\varepsilon(x,\cdot) \preceq C'\left(1+\frac{{\textrm{diam}}(M)^2}{\varepsilon}\right) \mathfrak{g}\,.
\ee

From this and \eqref{eq:hess:hess:cov:psieot} we deduce that
\[
\Hess_M \psi^n_\varepsilon \preceq \left(C'(1+\varepsilon^{-1}{\textrm{diam}}(M)^2) + C(1+\kappa^-(1+ \varepsilon)+\varepsilon^{-1})\right)\,\mathfrak{g}\,,
\]
so that the map $y\mapsto c_\varepsilon(x,y) + \psieot{n}(y)$ is $\Lambda$-semiconcave uniformly in $x\in\sSd$ and $n\in\N$ with 
\[
\Lambda = 2C'(1+\varepsilon^{-1}{\textrm{diam}}(M)^2) +  C(1+\kappa^-(1+ \varepsilon)+\varepsilon^{-1})\,.
\]
This allows us to apply \Cref{cor:exp_conv_from_stab} and deduce the exponential convergence of Sinkhorn's algorithm. Moreover,
as soon as
\be\label{eps:small:manifold:gen}
\varepsilon<\begin{cases}
+\infty\quad&\text{ if }\tau C\kappa^-\geq 1\,,\\
     \frac{2\tau C'+\tau C(1+\kappa^-)+\sqrt{(2\tau C'+\tau C(1+\kappa^-))^2+4\tau(1-\tau C\kappa^-)(2 C'{\textrm{diam}}(M)^2+C)}}{2(1-\tau C\kappa^-)}\quad&\text{ otherwise,}
\end{cases}
\ee
the exponential convergence rate equals
\bes
1-\frac{\varepsilon^2}{\varepsilon^2+  2\tau C'(\varepsilon+{\textrm{diam}}(M)^2) +  \tau C(\varepsilon+\kappa^-\varepsilon(1+ \varepsilon)+1) }\,.
\ees
Let us finally comment the case $M$ has non-negative Ricci curvature. A careful look at \cite{Stroock1996} shows that the restriction to $\varepsilon \leq 1$ is needed to ensure the validity of Eq.(2.8) therein, namely
\be\label{eq:gaussian_heat_kernel}
\frac{C_0}{{\textrm{vol}}(B_{\sqrt \varepsilon}(x))}e^{-C_1\frac{{\textrm{diam}}(M)^2}\varepsilon} \leq \msp_\varepsilon(x,y) \leq \frac{C_2}{{\textrm{vol}}(B_{\sqrt \varepsilon}(x))}\,, \qquad\forall x,y\in M,\, \forall \varepsilon \in (0,1]
\ee
for appropriate constants $C_0,C_1,C_2$ depending only on the curvature $\kappa$ and the dimension $\dim(M)$, from which he deduces our \eqref{doppio:hess:c:eps}. This is an adaptation to compact manifolds of Gaussian heat kernel estimates \cite{LiYau86}, paying attention to the fact that here $\msp_\varepsilon$ is the fundamental solution to $\partial_t u = \frac12 \Delta u$ instead of $\partial_t u = \Delta u$, whence the different constants in the exponential. Indeed, on a manifold $M$ with ${\textrm{Ric}}_M \succeq \kappa \mathfrak{g}$ it holds
\bes
\frac{C_0}{{\textrm{vol}}(B_{\sqrt \varepsilon}(x))}e^{-\frac{\sfd^2(x,y)}{\varepsilon} - C_\kappa \varepsilon} \leq \msp_\varepsilon(x,y) \leq  \frac{C_2}{{\textrm{vol}}(B_{\sqrt \varepsilon}(x))}e^{-\frac{\sfd^2(x,y)}{3\varepsilon} + C_\kappa \varepsilon}\,, \qquad\forall x,y\in M,\, \forall \varepsilon > 0
\ees
with $C_\kappa\geq 0$ solely depending on $\kappa$, and this implies both \cite[Eq.(2.8)]{Stroock1996} and \eqref{eq:gaussian_heat_kernel} if $\varepsilon \leq 1$. However, if ${\textrm{Ric}}_M \succeq 0$, the constant $C_\kappa$ in the bound above can be chosen equal to 0, whence the validity of \cite[Eq.(2.8)]{Stroock1996} and \eqref{eq:gaussian_heat_kernel} for all $\varepsilon>0$. As a consequence, the upper and lower estimates on $\Hess_M c_\varepsilon(x,\cdot)$ stated at~\eqref{doppio:hess:c:eps} holds true for all $\varepsilon>0$, whence the $\Lambda$-semiconcavity of $y\mapsto c_\varepsilon(x,y) + \psieot{n}(y)$ with the same $\Lambda$ as above. 
\begin{flushright}
    $\square$
\end{flushright}

\begin{remark}
To determine whether the condition $\tau C\kappa^- \geq 1$ in \eqref{eps:small:manifold:gen} is satisfied or not, we need a tractable expression for
\[
C=2\log \bigg(\frac{C_2C_3}{C_0}\bigg)\vee 4C_1 {\textrm{diam}}(M)^2
\]
namely for the constants $C_0,C_1,C_2$ appearing in the Gaussian heat kernel estimate \eqref{eq:gaussian_heat_kernel} and on the doubling constant $C_3$, see \eqref{eq:doubling}. If ${\textrm{Ric}}_M \succeq 0$ and we set $d = \dim(M)$, then $C_3 = 2^{d/2}$ while $C_0,C_1,C_2$ can be determined looking at the proof of \cite[Thm.\ 3.1, Cor.\ 3.1, and Thm.\ 4.1]{LiYau86}. Indeed, (applying the aforementioned results with $t/2$, choosing $\varepsilon=1$ and $\delta$ in such a way that $4+\varepsilon = 4(1+2\delta)(1+\delta)^2$ in \cite[Thm.\ 3.1, Cor.\ 3.1]{LiYau86} while $4-\varepsilon = 4(1-\delta)$ in \cite[Thm.\ 4.1]{LiYau86}) admissible values are $C_1 = 1$, $C_0 = C_2^{-1}$ and
\[
C_2 = \sup_{x,s>0} 4^d \left(\frac{4x}{\sqrt{s}}+1\right)^{d/2} e^{1-\frac{x^2}{30s}} \vee \frac{1}{(2\pi)^{d/2}} \,.
\]
\end{remark}

\subsection{Proof of~\Cref{prop_heavy_sink}}
It is shown at \cite[Theorem 3.1]{gentil2005modified} (see also short discussion thereafter) that the probability  measure $\tilde{\nu}_1$ on $\mathbb{R}$ whose density is proportional to $ e^{-|y|^p}$ satisfies a modified logarithmic Sobolev inequality. The tensorization result \cite[Proposition 2.3]{gentil2005modified} therein ensures that $\tilde{\nu}(\mathrm{d}y)\propto e^{-|y|_p^p}$ satisfies the same modified logarithmic Sobolev inequality. In the spirit of the Otto-Villani Theorem, it is shown at~\cite[Theorem 2.10]{gentil2005modified} that the modified logarithmic Sobolev inequality implies that $\tilde{\nu}$ satisfies the transportation inequality \ref{eq:generalized:TI} with $\omega(z,y)=L_{p,a}(|z-y|)$, for some finite positive constants $a>0$,
with the function $L_{p,a}$ being given by
\[L_{p,a}(r)=\begin{cases}\frac12r^2 \quad &r\leq a,\\ \frac{a^{2-p}}{p}r^{p}+a^2\frac{p-2}{2p}\quad &r\geq a.\end{cases}\]
In \cite[Corollary 1.7]{gozlan2011new} it has been shown that \ref{eq:generalized:TI} is stable under bounded perturbations.  Since $\frac{\mathrm{d}\nu}{\mathrm{d}\tilde{\nu}}$ is globally upper and lower bounded by construction, we conclude that there exists $\tau<+\infty$ (which may differ from the previous one) such that $\nu$ satisfies \ref{eq:generalized:TI} with $\omega(z,y)=L_{p,a}(|z-y|)$, \ie, that
\be\label{eq_talagrand_heavy_tails_for_nu}
\bfW_{L_{p,a}}(\mu,\nu)\leq \tau\,\mathrm{KL}(\mu|\nu) \quad \forall \mu\in\cP(\bbR^d).
\ee
We now proceed to quantify the regularity of of potentials. To this aim, we recall the terminology introduced in \cite{gozlan2025global}. We say that  a function $U$ is $\vartheta$-convex if 
\bes
U((1-t)x_0+tx_1)\leq (1-t)U(x_0)+tU(x_1)-t(1-t)\vartheta(|x_0-x_1|)\quad \forall\,t\in[0,1]\text{ and }\forall\,x_0,x_1\in\bbRD\,.
\ees 
Likewise, we say that $U$ is $\sigma$-smooth if 
\bes
U((1-t)x_0+tx_1)\geq (1-t)U(x_0)+tU(x_1)-t(1-t)\sigma(|x_0-x_1|)\quad \forall\,t\in[0,1]\text{ and }\forall\,x_0,x_1\in\bbRD\,.
\ees 
Owing to \cite[Prop 2.13]{gozlan2025global} we see that $U_\rho(x)=|x|^q+\delta|x|^2+\log Z_\rho$ (with $Z_\rho$ being the normalizing constant) is $\vartheta_\rho$-convex with $\vartheta_\rho(r)=\Lambda_\rho \max\{r^2,r^q\}$ and $\Lambda_\rho$ some positive constant. Similarly, from \cite[Section 2.3]{gozlan2025global} we know that $U_\nu(y)=\min\{r^2,r^p\}+\log Z_\nu$ is $\sigma_\nu$-smooth with $\sigma_\nu(r)=\Lambda_{\nu}\min\{r^2,r^p\}$ and $\Lambda_\sigma$ some positive constant. We now show a uniform smoothness bound on Sinkhorn's iterates. The proof is an adaptation of the arguments used to establish Theorem 4.2 in \cite{gozlan2025global}.

\begin{proposition}\label{prop_gozlan_ping_pong}
    In the same setting and assumptions of Proposition \ref{prop_heavy_sink}  there exists $\Lambda_\psi<+\infty$ such that $\frac{1}{2}|\cdot-x|^2+\psieot{n}$ is $\Lambda_{\psi} L_{a,p}$-smooth uniformly in $n$ and $x$.
\end{proposition}

\begin{proof}
   In order to be able to invoke the results of \cite{gozlan2025global} we introduce the functions $\tilde{\psi}^n_\varepsilon=\psieot{n}+\frac12|\cdot|^2$ and $\tilde{\varphi}^n_\varepsilon=\phieot{n}+\frac12|\cdot|^2$. With this notation at hand, we see that \eqref{eq:sink:iterate} can be reformulated as
    \bes
\begin{cases}
\tilde{\varphi}^{n+1}_\varepsilon &= \cL_{\varepsilon,\Leb} (\tilde{\psi}^n_\varepsilon+\varepsilon\,U_\nu)\,,\\
\tilde{\psi}^{n+1}_{\varepsilon} &=\cL_{\varepsilon,\Leb} (\tilde{\varphi}^{n+1}_\varepsilon+\varepsilon\,U_\rho)\,,
\end{cases}
\ees
with the entropic Legendre transform operator (w.r.t. the Lebesgue measure $\Leb$) being defined as
\bes
\cL_{\varepsilon,\Leb}(h)(x)\coloneqq \varepsilon\log\int\exp\biggl(\frac{\langle x,\,y\rangle -h(y)}{\varepsilon}\biggr)\De y\,.
\ees
We are going to prove by induction that $\tilde{\psi}^{n}_\varepsilon$ is $\sigma$-smooth with 
\be\label{eq_sigma_repr}
\sigma(r) = \int_0^r \vartheta_{\rho}^{-1}(\sigma_\nu(s)) \mathrm{d}s,
\ee
and $\vartheta_\rho,\sigma_\nu$ as in the discussion above.
We will show later that $\sigma$ can be controlled by $L_{a,p}$ up to a multiplicative constant in order to conclude the proof. The base case $n=0$ holds true by assumption. Next, let us  assume that   $\tilde{\psi}^{n}_\varepsilon$ is $\sigma$-smooth. 
Then \cite[Proposition 3.2]{gozlan2025global} implies that $\tilde{\varphi}^{n+1}_\varepsilon = \cL_{\varepsilon,\Leb} (\tilde{\psieot{n}}+\varepsilon\,U_\nu)$ is $(\sigma+\varepsilon\sigma_\nu)^\star$-convex, where for any $\alpha\colon\rset_+\to\rset\cup\{+\infty\}$ the function $\alpha^\star$ denotes its monotone conjugate defined as
\bes
\alpha^\star(u)\coloneqq \sup_{t\geq 0}\{tu-\alpha(t)\}\,.
\ees
Invoking now \cite[Proposition 3.1]{gozlan2025global} we deduce that $\tilde{\psi}^{n+1}_{\varepsilon} =\cL_{\varepsilon,\Leb} (\tilde{\varphi}^{n+1}_\varepsilon+\varepsilon\,U_\rho)$ is $((\sigma+\varepsilon\sigma_\nu)^*+\varepsilon \vartheta_\rho)^*$-smooth. Therefore, in order to conclude the proof of the inductive step we need to show that $((\sigma+\varepsilon\sigma_\nu)^*+\varepsilon \vartheta_\rho)^*\leq \sigma$, \ie, that for any $r,t\geq 0$ 
\bes
tr-(\sigma+\varepsilon\sigma_\nu)^*(t)-\varepsilon \vartheta_\rho(t)\leq \sigma(r)
\ees
Equivalently, we need to show that for all $t,\,r\geq0$ there exists some $s>0$ such that 
\be\label{eq_fixed_point_profile}
tr-ts+\sigma(s)+\varepsilon\sigma_\nu(s)-\varepsilon \vartheta_\rho(t)\leq \sigma(r)\,.
\ee
In view of this, notice that  $\sigma'(u)$ is strictly increasing and thus invertible on $[0,+\infty)$. Therefore, for a given pair $t,r$ we can pick $s=(\sigma')^{-1}(t)$. With this choice \eqref{eq_fixed_point_profile} becomes
\bes
\varepsilon\sigma_\nu(s)-\varepsilon \vartheta_\rho(\sigma'(s))\leq \sigma(r)-\sigma(s)-\sigma'(s)(r-s)\,.
\ees
By recalling the definition of $\sigma$ from \eqref{eq_sigma_repr} we see that $\sigma'=(\vartheta_\rho)^{-1}(\sigma_\nu)$ which implies that the previous claimed statement is equivalent to 
\bes
0\leq \sigma(r)-\sigma(s)-\sigma'(s)(r-s)\,,
\ees
which follows by convexity of $\sigma$. We have thus proven that   $\psieot{n}+\frac12|\cdot|^2$ is $\sigma-$smooth for all $n\geq1$. From this, it easily follows that $\psieot{n}+\frac12|\cdot-x|^2$ is $\sigma-$smooth for all $n\geq1$ and all $x\in \bbR^d$. To conclude we show that $\sigma \leq \Lambda_\psi L_{a,p}$ for some $\Lambda_\psi<+\infty$. To this aim we recall that $\theta_\rho(r) = \Lambda_\rho\max\{r^2,r^q\},\sigma_\nu(r) = \Lambda_\nu\min\{r^2,r^p\}$. From this, we obtain $\theta^{-1}_\rho(r) \leq \min\{(r/\Lambda_\rho)^{1/2},(r/\Lambda_\rho)^{1/q}\}$.  Observing that we can w.l.o.g. assume that $\Lambda_\rho\leq\Lambda_\nu$ we obtain
\begin{equation*}
\sigma(r) \leq \begin{cases}\frac12\sqrt{\frac{\Lambda_\nu}{\Lambda_\rho}}r^2, \quad &r\leq \Lambda_\nu,\\  \frac{\Lambda^{-1/q}_\rho\Lambda_\nu^{p/q}}{p/q+1}(r^{p/q+1}-\Lambda^{p/q+1}_\nu)+\frac12\frac{\Lambda^{5/2}_\nu}{\Lambda_\rho^{1/2}}, \quad &r\geq \Lambda_\nu,\end{cases}
\end{equation*}
where to bound $\sigma$ in the interval $[0,\Lambda_\nu]$ we use the bounds $\vartheta_\rho^{-1}(r)\leq (r/\Lambda_{\rho})^{1/2}$ and $\sigma_\nu(r)\leq \Lambda_\nu\,r^2$, whereas to bound $\sigma$ on $[\Lambda_\nu,+\infty]$ we use the bounds  $\vartheta_\rho^{-1}(r)\leq (r/\Lambda_{\rho})^{1/q}$ and $\sigma_\nu(r)\leq \Lambda_\nu\,r^p$.
Since $p/q+1\leq p$ by assumption, it follows $\sigma$ is bounded by $L_{a,p}$ up to a multiplicative constant for all $a>0$.
\end{proof}
\begin{proof}[Proof of Proposition \ref{prop_heavy_sink}]
Proposition \ref{prop_gozlan_ping_pong} gives that $\psieot{n}+\frac{1}{2}|\cdot-x|^2$ is $\Lambda_\psi L_{a,b}$- smooth for some $\Lambda_\psi,a$ uniformly on $n\geq1,\varepsilon>0,x\in\bbR^d$. Because of \cite[Lemma 2.4]{gozlan2025global} we also get that $\psieot{n}+\frac{1}{2}|\cdot-x|^2$ is $(\Lambda_\psi,L_{a,b})$- semiconcave. Moreover, we have shown at \eqref{eq_talagrand_heavy_tails_for_nu} that $\nu$ satisfies \ref{eq:generalized:TI} for $\omega=L_{a,b}$ and $\gamma=1$. We have verified the hypothesis of Theorem \ref{thm:generalized:sink}-\ref{item:exp_conv_from_stab_i:gen}, which gives the desired result.
\end{proof}


\bibliographystyle{alpha}
\bibliography{./X2main.bbl}

\appendix

\section{Sufficient conditions for \Cref{ass:differentiability}}\label{sec:app:integrability:weak}

In this Appendix we provide an integrability condition for $\rho$ (involving the cost function $c$ and its first derivative in $y\in\cY$) which is sufficient to guarantee the validity of the differentiation assumption \Cref{ass:differentiability}.
\begin{lemma}\label{lem_der_under_integral_sign}
Assume $c(x,\cdot)$ to be $\cC^1(\cY)$ and $c(x,\cdot)+\psieot{\nu}(\cdot)$ $\Lambda$-semiconcave uniformly in $x\in\cX$.
If for all $\delta>0$ and all open balls $B\subseteq \cY$ we have
\begin{equation*}
\int (1+\sup_{y\in B}|\nabla_2c(x,y)|)\exp\biggl(\delta\sup_{y\in B}|\nabla_2c(x,y)|\biggr)\rho(\mathrm{d}x)<+\infty,
\end{equation*}
then the differentiation formula 
\begin{equation*}
\nabla\psieot{\nu}(y)=-\int_\cX\nabla_2 c(x,y)\,\pieot{\nu}(\De x|y)\,,
\end{equation*}
holds in the sense of weak derivatives.
\end{lemma}
\begin{proof}
Recall the definition of weak derivative: we say that $g=\nabla f$ in the sense of weak derivatives if
\begin{equation*}
\int_{\cY} f\, (\nabla\cdot h)(y) \,\,\mathrm{d}y  = \int_{\cY} g\cdot h(y) \,\,\mathrm{d}y, \quad \forall h\in\cC^{\infty}_c(\cY;\bbR)
\end{equation*}
We first observe that, being $c(x,\cdot) + \psi^\nu_\varepsilon(\cdot)$ $\Lambda$-semiconcave by assumption and with values in $\bbR$, $\psi^\nu_\varepsilon(\cdot)$ is locally Lipschitz. As such, both $\psi^\nu_\varepsilon(\cdot)$ and $\nabla\psi^\nu_\varepsilon(\cdot)$ are locally integrable.
Therefore, using the same arguments as in \cite[Appendix A]{conforti2023projected}, in order to conclude, it suffices to show that for any ball $B\subseteq \cY$ we have
\begin{equation}\label{eq:suff_cond_from_Lacker}
\begin{split}
\iint\mathbf{1}_{B}(y)\exp(-(c(x,y)+\varphi(x))/\varepsilon)\rho(\mathrm{d}x) \De y<+\infty, \\
\iint\mathbf{1}_{B}(y)|\nabla_2c(x,y)|\exp(-(c(x,y)+\varphi(x))/\varepsilon)\rho(\mathrm{d}x)\De y<+\infty.
\end{split}
\end{equation}
By applying Jensen's inequality in the Schrödinger system \eqref{eq:schr_syst} we obtain
\begin{equation*}
\phieot{\nu}(x) \geq  -\int c(x,y)+\psieot{\nu}(y)\,\,\nu(\De y).
\end{equation*}
Now, fix any $y_0\in B$ such that $c(x,\cdot)+\psieot{\nu}(\cdot)$ is differentiable at $y_0$.
As a consequence, of our smoothness assumption on $c$, this means that $\psieot{\nu}$ is also differentiable at $y_0$.  
But then, for any $y\in\cY$ we may consider $(\gamma^y_t)_{t\in[0,1]}$ geodesic from $y_0$ to $y$  and from the $\Lambda$-semiconcavity deduce that
\bes
\begin{split}
 \int c(x,y)&+\psieot{\nu}(y)\,\,\nu(\De y) \leq 
c(x,y_0)+\psieot{\nu}(y_0) +\int\biggl[ \langle \nabla_2c(x,y_0)+\nabla\psieot{\nu}(y_0),\,\dot{\gamma}^y_0\rangle_{\mathfrak{g}}+\frac{\Lambda}{2}\sfd^2(y_0,y)\biggr]\,\nu(\De y)\\
\leq&\,c(x,y_0)+\psieot{\nu}(y_0)+|\nabla_2c(x,y_0)+\nabla\psieot{\nu}(y_0)|\,\int\sfd(y_0,y)+\frac{\Lambda}{2}\sfd^2(y_0,y)\,\nu(\De y)\\
\leq & c(x,y_0) +c_0 (1+|\nabla_2c(x,y_0)|)
 \end{split}
\ees
where $c_0$ is a finite positive constant independent of $x$ whose value may change from line to line. Therefore, we have 
\begin{equation*}
    c(x,y)+\phieot{\nu}(x) \geq -|c(x,y)-c(x,y_0)| -c_0 (1+|\nabla_2c(x,y_0)|) \geq -c_0 (1+\sup_{y\in B}|\nabla_2c(x,y)|)
\end{equation*}
But then,
\begin{align*}
\iint&\mathbf{1}_{B}(y) (1+|\nabla_2c(x,y)|)\,e^{-\frac{c(x,y)+\phieot{\nu}(x)}{\varepsilon}}\rho(\mathrm{d}x)\De y \\
\leq & e^{-\nicefrac{c_0}{\varepsilon}}\mathrm{vol}(B)\int (1+\sup_{y\in B}|\nabla_2c(x,y)|)\exp\biggl(\frac{c_0}{\varepsilon}\sup_{y\in B}|\nabla_2c(x,y)|\biggr)\rho(\mathrm{d}x)\,,
\end{align*}
from which the identities \eqref{eq:suff_cond_from_Lacker} follows thanks to our assumption.
\end{proof}

\section{A polynomial convergence result for Sinkhorn's algorithm}\label{appendix:numerica:iteration}

In this appendix we show how our $(\Lambda,\omega)$-semiconcavity approach might actually be employed to establish polynomial convergence rates for Sinkhorn's algorithm.

\begin{theorem}[Polynomial convergence in $(\Lambda,\omega)$-semiconcave settings]\label{thm:generalized:sink:polynomial}
    Assume $c(x,\cdot)$ to be $\cC^1(\cY)$, that there exist $\Lambda\in(0,+\infty)$, a non-negative function $\omega\colon\cY\times\cY\to\rset$ and $N\geq 2$ such that 
\bes
y\mapsto c(x,y)+\psieot{n}(y)
\ees
is $(\Lambda,\omega)$-semiconcave uniformly in $x\in\supp(\rho)$ and $n\geq N-1$. 
If $\nu$ satisfies the generalized transport inequality \eqref{eq:generalized:TI:gamma} for some $\gamma\in(0,1)$  then for any $n\geq N$ we have
\bes
\scrH(\pieot{\nu}|\pi^{n+1,n})\leq (n-N+1)^{-\frac{\gamma}{1-\gamma}} \biggl(\frac{1-\gamma}{\gamma M^{\frac1\gamma}}+\frac1{(n-N+1)\,\scrH(\pieot{\nu}|\pi^{N,N-1})^\frac{1-\gamma}{\gamma}}\biggr)^{-\frac{\gamma}{1-\gamma}}\,,
\ees
where for notation's sake we have introduced the constant 
\bes
M=M\biggl(\gamma,\frac{\tau\Lambda}{2\varepsilon}\biggr)\coloneqq 2^{1-\gamma}\max\biggl\{\scrH(\nu|\nu^{N,N-1})^{1-\gamma},\,\frac{\tau\Lambda}{2\varepsilon}(1+\scrH(\nu^{N,N-1}|\nu)^{1-\gamma})\biggr\}\,.
\ees
\end{theorem}

\begin{proof}
Let $n\geq N$. From Theorem \ref{thm:entropic_stability:gen}, applied to the pairs of marginals $(\rho,\nu)$ and $(\rho,\nu^{n+1,n })$, we obtain 
\bes
\begin{split}
\scrH(\pieot{\nu}|&\pi^{n+1,n}) \leq\,\scrH(\nu|\nu^{n+1,n})+\frac{\Lambda}{2\varepsilon} \bfW_\omega(\nu^{n+1,n},\nu)\\
\stackrel{\eqref{eq:generalized:TI}}{\leq}&\, \scrH(\nu|\nu^{n+1,n})+\frac{\tau\Lambda}{2\varepsilon}[\scrH(\nu^{n+1,n}|\nu)+\scrH(\nu^{n+1,n}|\nu)^\gamma]\\
\leq&\, \max\{\scrH(\nu|\nu^{n+1,n})^{1-\gamma},\,\nicefrac{\tau\Lambda}{2\varepsilon}(1+\scrH(\nu^{n+1,n}|\nu)^{1-\gamma})\}\big( \scrH(\nu^{n+1,n}|\nu)^\gamma+\scrH(\nu|\nu^{n+1,n})^\gamma\big)\,.
\end{split}
\ees
Invoking Sinkhorn's monotonicity inequalities \cite[Proposition 6.10]{Marcel:notes}
\bes
\scrH(\nu^{n+1,n}|\nu) \leq \scrH(\rho|\rho^{n,n})\,, \quad \scrH(\nu|\nu^{n+1,n})\leq \scrH(\nu|\nu^{n,n-1})\,,
\ees
we arrive at the following bound
\bes
\begin{aligned}
\scrH(&\pieot{\nu}|\pi^{n+1,n}) \\
\leq&\, \max\biggl\{\scrH(\nu|\nu^{N,N-1})^{1-\gamma},\,\frac{\tau\Lambda}{2\varepsilon}(1+\scrH(\nu^{N,N-1}|\nu)^{1-\gamma})\biggr\}(\scrH(\rho|\rho^{n,n})^\gamma+\scrH(\nu|\nu^{n,n-1})^\gamma)\\
\leq &\,M(\scrH(\rho|\rho^{n,n})+\scrH(\nu|\nu^{n,n-1}))^\gamma\,,
\end{aligned}
\ees
where we have set $M\coloneqq 2^{1-\gamma}\max\biggl\{\scrH(\nu|\nu^{N,N-1})^{1-\gamma},\,\frac{\tau\Lambda}{2\varepsilon}(1+\scrH(\nu^{N,N-1}|\nu)^{1-\gamma})\biggr\}$ for notation's sake.
Using this result in \eqref{entropy_difference_along_sink} gives 
\be\label{eq:gen:iterate:gamma}
\begin{aligned}
\scrH(\pieot{\nu}|\pi^{n+1,n})-\scrH(\pieot{\nu}|\pi^{n,n-1}) 
=-(\scrH(\rho|\rho^{n,n})+\scrH(\nu|\nu^{n,n-1}))\leq -\frac{\scrH(\pieot{\nu}|\pi^{n+1,n})^{\nicefrac1\gamma}}{M^{\nicefrac1\gamma}}\,.
\end{aligned}
\ee
This simple recursion in particular implies (\Cref{lemma:appendix:numerical:iterates} below) that for all $n\geq N$
\bes
\scrH(\pieot{\nu}|\pi^{n+1,n})\leq (n-N+1)^{-\frac{\gamma}{1-\gamma}} \biggl(\frac{1-\gamma}{\gamma M^{\frac1\gamma}}+\frac1{(n-N+1)\,\scrH(\pieot{\nu}|\pi^{N,N-1})^\frac{1-\gamma}{\gamma}}\biggr)^{-\frac{\gamma}{1-\gamma}}\,.
\ees
\end{proof}

Below we state and prove the technical lemma used in order to conclude the proof of \Cref{thm:generalized:sink:polynomial} (there we set $\alpha=\gamma^{-1}>1$, $a_n=\scrH(\pieot{\nu}|\pi^{n+1,n})$ and $C=M^{\frac1\gamma}$). 
\begin{lemma}\label{lemma:appendix:numerical:iterates}
    Let $(a_n)_{n\in\N}$ be a sequence of non-negative reals and suppose for any $n\geq N$ it satisfies
    \bes
    a_n-a_{n-1}\leq -\frac{a_n^\alpha}{C}\,,
    \ees
    for some positive constant $C>0$ and some $\alpha>1$. Then for any $n\geq N$ we have
    \bes
    a_n\leq (n-N)^{-\frac{1}{\alpha-1}}\,\biggl(\frac{\alpha-1}{C}+\frac1{(n-N)a_N^{\alpha-1}}\biggr)^{-\frac{1}{\alpha-1}}\,.
    \ees
\end{lemma}
\begin{proof}
 Firstly, notice that the sequence $(a_n)_{n\geq N}$ is non-increasing since we have $a_{n-1}\geq a_n+C^{-1}a_n^\alpha\geq a_n$ and therefore the limit exists (since it is non-negative). Next, introduce the convex function $f(x)=x^{-(\alpha-1)}$ and let $b_n\coloneqq f(a_n)$. From the convexity of $f$ we then have for any $n\geq N$
 \bes
 b_n-b_{n-1}=f(a_n)-f(a_{n-1})\geq (\alpha-1)a_{n-1}^{-\alpha}(a_{n-1}-a_n)\geq \frac{\alpha-1}{C}\,.
 \ees
 From this we see that for any $n\geq N$
 \bes
b_n=b_N+\sum_{k=N}^{n-1} b_{k+1}-b_k\geq a_N^{-(\alpha-1)}+(n-N)\frac{\alpha-1}{C}\,,
 \ees
 and hence that
 \bes
a_n=b_n^{-\frac1{\alpha-1}}\leq \biggl((n-N)\frac{\alpha-1}{C}+a_N^{-(\alpha-1)}\biggr)^{-\frac1{\alpha-1}}\,,
 \ees
 from which our thesis follows.
\end{proof}

\section{On the propagation of convexity}\label{appendix}
We write here a corollary of \cite[Lemma 3.1]{conforti2024weak} using the notation of this article. Recall that here we consider the quadratic cost $c(x,y)=\|x-y\|^2/2$ in $\bbRD$.

\begin{lemma}
Let $\varphi$ be such that for some $\alpha>0,L\geq0$ 
\bes
\kappa_{(\varepsilon^{-1}\varphi+U_\rho)}(r)\geq \alpha-\varepsilon^{-1}-r^{-1} f_L(r)\, \quad \forall r>0
\ees
holds. Then we have
\bes
\kappa_{\Phi^\rho_0(\varphi)}(r)\geq \frac{\varepsilon \alpha -1}{\varepsilon\alpha}- \frac{L}{\varepsilon\alpha^2} \quad \forall r>0.
\ees
\end{lemma}

To connect the notation used here with that of \cite{conforti2024weak} we observe that for all $\varphi$, we have that, upon setting $g=\varepsilon^{-1}\varphi+U_{\rho}$, we have $\Phi^\rho_0(\varphi)=\varepsilon\, U^{\varepsilon,g}_0$, where $[0,\varepsilon]\times\bbR^d\ni(s,x)\mapsto U^{\varepsilon,g}_s(x)$ is the solution of the Hamilton-Jacobi-Bellman equation
\bes
\bec 
\partial_s \theta_s+\frac12\Delta\theta_s-\frac12|\nabla\theta_s|^2=0,\\
\theta_\varepsilon=g.
\eec
\ees

\section{Proof of Lemma \ref{lem:hamilton_heatkernel}}\label{appendix2}

As a starting point, recall Hamilton's gradient estimate \cite{Hamilton93}: this ensures that if $u$ is a positive solution to the heat equation $\partial_t u = \Delta u$ with $u(0,\cdot) \in \rmL^\infty(M)$ and ${\textrm{Ric}}_M \succeq \kappa\mathfrak{g}$ globally on $M$ for some $\kappa \in \mathbb{R}$, then
\be 
t|\nabla_M \log u(t,\cdot)|^2 \leq (1+2\kappa^- t)\log\left(\frac{\|u(0,\cdot)\|_{\rmL^\infty}}{u(t,\cdot)}\right), \qquad \forall t > 0\,.
\ee
The desired bound \eqref{eq:gradient_heat_kernel} will follow by applying the inequality above to $u(s,y)=\msp_{t+s}(x,y)$, where $t>0$ and $x \in M$ are fixed. To this end, we need to verify that $u(0,\cdot) = \msp_t(x,\cdot) \in \rmL^\infty(M)$, since by the compactness of $M$ we already know that ${\textrm{Ric}} \geq \kappa\Id$ for some $\kappa \in \mathbb{R}$. In particular, the fact that the Ricci curvature tensor is lower bounded grants \eqref{eq:gaussian_heat_kernel}, whence
\[
\begin{split}
\|\msp_t(x,\cdot)\|_{\rmL^\infty} & = \sup_{y \in M}  \msp_t(x,y) \leq \frac{C_2}{{\textrm{vol}}(B_{\sqrt t}(x))} \,, \\
\inf_x \msp_{2t}(x,y) & \geq \frac{C_0e^{-C_1\frac{{\textrm{diam}}(M)^2}t}}{{\textrm{vol}}(B_{\sqrt{2t}}(x))} > 0 \,.
\end{split}
\]
By the Bishop--Gromov inequality together with the compactness of $M$ (see for instance \cite{petersen2006}), we know that for some constant $C_3>0$ it holds
\be\label{eq:doubling}
{\textrm{vol}}(B_{\sqrt{2t}}(x)) \leq C_3 {\textrm{vol}}(B_{\sqrt{t}}(x))\,, \qquad\forall x\in M,\, \forall t>0\,,
\ee
so that the above yields
\[
\frac{\|\msp_t(x,\cdot)\|_{\rmL^\infty}}{\msp_{2t}(x,y)} \leq \frac{C_2 C_3}{C_0}e^{C_1\frac{{\textrm{diam}}(M)^2}t}\,, \qquad \forall x,\,y \in M,\, \forall t \in (0,1]\,.
\]
We can now apply Hamilton's gradient estimate with $u(s,y)=\msp_{t+s}(x,y)$, as anticipated, and then set $s=t$, to get
\[
\begin{aligned}
t|\nabla_M\log\msp_{2t}(x,\cdot)(y)|^2 &\leq (1+2\kappa^-t)\log\left(\frac{\|\msp_t(x,\cdot)\|_{\rmL^\infty}}{\msp_{2t}(x,y)}\right) \\&\leq (1+2\kappa^-t) \left(\log\bigg(\frac{C_2C_3}{C_0}\bigg) +C_1\frac{{\textrm{diam}}(M)^2}t\right)\,,
\end{aligned}
\]
which is readily verified to be equivalent to \eqref{eq:gradient_heat_kernel} with $C=2\log \big(\frac{C_2C_3}{C_0}\big)\vee 4C_1 {\textrm{diam}}(M)^2$.
\begin{flushright}
    $\square$
\end{flushright}

\end{document}